\documentclass[11pt]{article}
\title{\textsc{Notions of maximality for integral lattice-free polyhedra:\\ the case of dimension three}}
\author{Gennadiy Averkov\footnote{Faculty of Mathematics, University of Magdeburg, Universit\"atsplatz 2, 39106 Magdeburg, Germany.
Emails: averkov@ovgu.de, jan.kruempelmann@ovgu.de, weltge@ovgu.de} \qquad Jan Kr\"umpelmann$^{\ast}$ \qquad Stefan Weltge$^{\ast}$}

\usepackage{etex}
\usepackage{graphicx,subfigure}
\usepackage{pstricks}
\usepackage{pst-all,pst-eps}
\usepackage{tikz}
\usepackage{todonotes}
\usepackage{diagbox}
\usepackage{array}
\usepackage{url}
\usepackage{paralist}
\usepackage[boxed,linesnumbered]{algorithm2e}

\DeclareFixedFont{\ttb}{T1}{txtt}{bx}{n}{9} 
\DeclareFixedFont{\ttm}{T1}{txtt}{m}{n}{9}  
\DeclareFixedFont{\tti}{T1}{txtt}{m}{it}{9}  

\usepackage[procnames]{listings}

\usepackage{pifont}

\usepackage{amsmath,amsthm,amssymb,enumerate}
\usepackage{enumitem}
\newcommand{\rmcmd}[1]{\mathop{\mathrm{#1}}\nolimits}
\newcommand{\aff}{\rmcmd{aff}}

\newcommand{\R}{\mathbb{R}}
\newcommand{\N}{\mathbb{N}}
\newcommand{\Z}{\mathbb{Z}}

\newcommand{\cP}{\mathcal{P}}

\newcommand{\setcond}[2]{\left\{#1 \, : \, #2 \right\}}
\newcommand{\setcondsep}{:}
\newcommand{\sprod}[2]{\left< #1 \, , \, #2 \right>}

\newcommand{\textand}{\text{and}}

\newcommand{\vol}{\rmcmd{vol}}

\newcommand{\ld}[1]{\, \operatorname{ld}(#1)\,}
\newcommand{\lw}[1]{\, \operatorname{lw}(#1)\,}
\newcommand{\width}[1]{\, \operatorname{w}(#1)\,}

\newcommand{\vertset}[1]{\, \operatorname{vert}(#1)}
\newcommand{\conv}{\operatorname{conv}}
\newcommand{\cone}{\operatorname{cone}}
\newcommand{\intrbig}[1]{\operatorname{int}\big(#1\big)}
\newcommand{\intr}[1]{\operatorname{int}(#1)}
\newcommand{\relintrbig}[1]{\operatorname{relint}\big(#1\big)}
\newcommand{\relintr}[1]{\operatorname{relint}(#1)}
\newcommand{\bd}[1]{\operatorname{bd}(#1)}
\newcommand{\relbd}[1]{\operatorname{relbd}(#1)}

\newcommand{\lspan}{\operatorname{lin}}

\newcommand{\oq}[1]{\overline{P}_{#1}}

\newcommand{\cL}{\mathcal{L}}

\newcommand{\tq}{\widetilde{Q}}
\newcommand{\tp}{\widetilde{P_0}}
\newcommand{\fixedLatticeDiameter}{\ell}
\newcommand{\fixedBase}{B}
\newcommand{\fixedApex}{a}
\newcommand{\fixedPyramid}{T}
\newcommand{\candidates}[1]{\operatorname{cand}(#1)}
\newcommand{\safeRegion}[1]{\operatorname{s}(#1)}

\newcommand{\computerFinalSet}{S}
\newcommand{\maximizer}{M}

\newcommand{\computerIntermediateSet}[1]{S_{#1}}

\newtheoremstyle{ourtheoremstyle} 
    {\topsep}                    
    {\topsep}                    
    {\itshape}                   
    {}                           
    {\scshape}                   
    {.}                          
    {.5em}                       
    {}  
\newtheoremstyle{ourremarkstyle} 
    {\topsep}                    
    {\topsep}                    
    {}                   
    {}                           
    {\scshape}                   
    {.}                          
    {.5em}                       
    {}  

\theoremstyle{ourtheoremstyle}
\newtheorem{theorem}{Theorem}

\newtheorem{proposition}[theorem]{Proposition}
\newtheorem{lemma}[theorem]{Lemma}
\newtheorem{corollary}[theorem]{Corollary}

\newtheorem*{question*}{Question}
\theoremstyle{ourremarkstyle}
\newtheorem{remark}[theorem]{Remark}

\usetikzlibrary{calc,3d}
\usetikzlibrary{decorations.pathreplacing}

\AtBeginDocument{%
   \def\MR#1{}
}

\begin{document}
\maketitle

\begin{abstract}
Lattice-free sets (convex subsets of $\R^d$ without interior integer points) and their applications for cutting-plane methods
in mixed-integer optimization have been studied in recent literature. Notably, the family of all integral lattice-free polyhedra
which are not properly contained in another integral lattice-free polyhedron has been of particular interest. We call
these polyhedra $\Z^d$-maximal.

It is known that, for fixed $d$, the family $\Z^d$-maximal integral lattice-free polyhedra is finite up to unimodular equivalence.
In view of possible applications in cutting-plane theory, one would like to have a classification of this family. However, this turns out
to be a challenging task already for small dimensions.

In contrast, the subfamily of all integral lattice-free polyhedra which are not properly contained in any other lattice-free set,
which we call $\R^d$-maximal lattice-free polyhedra, allow a rather simple geometric characterization. 
Hence, the question was raised for which dimensions the notions of $\Z^d$-maximality and $\R^d$-maximality are equivalent. This was 
known to be the case for dimensions one and two.
On the other hand, Nill and Ziegler (2011) showed that for dimension $d \ge 4$, there exist polyhedra which are $\Z^d$-maximal but not $\R^d$-maximal.
In this article, we consider the remaining case $d = 3$ and prove that for integral polyhedra the notions of $\R^3$-maximality and $\Z^3$-maximality are equivalent.
As a consequence, the classification of all $\R^3$-maximal integral polyhedra by Averkov, Wagner and Weismantel (2011)
contains all $\Z^3$-maximal integral polyhedra.
\end{abstract}

\newtheoremstyle{itsemicolon}{}{}{\mdseries\rmfamily}{}{\itshape}{:}{ }{}
\newtheoremstyle{itdot}{}{}{\mdseries\rmfamily}{}{\itshape}{:}{ }{}
\theoremstyle{itdot}
\newtheorem*{msc*}{2010 Mathematics Subject Classification} 

\begin{msc*}
	52B10, 52B20, 52C07, 90C11
\end{msc*}

\newtheorem*{keywords*}{Keywords}

\begin{keywords*}
classification; cutting planes; integral polyhedra; lattice-free sets;
mixed-integer optimization
\end{keywords*}

\section{Introduction}

We call a $ d $-dimensional convex subset $C$ of $\R^d$ \emph{lattice-free} if the interior of $C$ contains no points of
$\Z^d$.
For a subset $X$ of $\R^d$, we call a lattice-free convex set $C$ {\em $X$-maximal} if for every $x \in X \setminus C$, the
set $\conv(C \cup \{x\})$ is not lattice-free.
For $Y \subseteq \R^d$, we say that $P$ is a $Y$-polyhedron if $P = \conv(P \cap Y)$.
A $\Z^d$-polyhedron is also called an {\em integral} polyhedron.
We denote the family of all $Y$-polyhedra by $\cP(Y)$.
In particular, $\cP(\R^d)$ is the family of all polyhedra in $\R^d$ and $\cP(\Z^d)$ is the family of all integral
polyhedra in $\R^d$.

Recently, the family of $\Z^d$-maximal integral lattice-free polyhedra has attracted attention of experts in algebraic
geometry and optimization; see~\cite{Treutlein08,Treutlein10}, \cite{MR2832401} and \cite{MR2855866}.
With a view towards applications in the mentioned research areas, having a better geometric understanding of such
polyhedra is desirable.
Currently, we have a rather good geometric description of $\R^d$-maximality (which is a stronger property than $ \Z^d
$-maximality) for general lattice-free sets; see \cite{MR1114315} and \cite{MR3027668}.
In contrast, so far no simple description of $\Z^d$-maximality is available, even in the case of integral polyhedra.
This has motivated the following question, asked in \cite{MR2832401} and \cite{MR2855866}: is every $\Z^d$-maximal
integral lattice-free polyhedron also $\R^d$-maximal?
This can also be formulated as follows: is it true for every integral lattice-free polyhedron $P$, that if $\conv(P\cup
\{x\})$ is lattice-free for some $x \in \R^d \setminus P$, then $\conv(P \cup \{x\})$ is lattice-free for some $x \in
\Z^d \setminus P$?
The answer to this question is trivially `yes' for dimensions $d=1,2$, while it was shown in~\cite{MR2832401} that the answer is `no' for
each $d \ge 4$.
The purpose of this paper is to settle the case $d=3$, which remained open; see~\cite[Question~1.5]{MR2832401} and \cite[Section 7.4]{wagner_diss}.

\begin{theorem}\label{thm:main_three}
Every $\Z^3$-maximal integral lattice-free polyhedron is also $\R^3$-maximal.
\end{theorem}

The interest in the family of lattice-free integral polyhedra, and especially those that are $\Z^d$-maximal, was raised
by potential applications in mixed-integer optimization, more precisely in cutting-plane theory.
It is known that lattice-free sets can be used in mixed-integer programming for describing various families of cutting planes. 
Consider the set $P \cap (\Z^d \times \R^n)$ of feasible solutions of an arbitrary mixed-integer linear program, where $P \subseteq \R^{d+n}$ is a 
rational polyhedron, $d \in \N$ is the number of integer variables and $n \in \N \cup \{0\}$ is the number of continuous variables. 
Given a lattice-free set $L \subseteq \R^d$,
we call a closed halfspace $H$ of $\R^{d+n}$ an \emph{$L$-cut} for the polyhedron $P$ if $P \setminus \intr{L} \times \R^n \subseteq H$. 
Furthermore, given a family $\cL$ of lattice-free subsets of $\R^d$, we say that $H$ is an {\em $\cL$-cut} if $H$ is an $L$-cut for some $L \in \cL$. 

\newcommand{\MI}{{\operatorname{MI}}}

In the theory of general mixed-integer linear programs, one is interested in explicit descriptions of families $\cL$ being possibly simple and small on the one hand, and generating strong cuts on the other hand. The strength of $\cL$-cuts can be represented in terms of the sequence of $\cL$-closures, defined as follows. For a rational polyhedron $P \subseteq \R^{d+n}$, we define the \emph{$\cL$-closure} of $P$ by $c_\cL(P) = \bigcap \setcond {P \cap H}{H \ \text{is an $\cL$-cut}}$. For a nonnegative integer $k$, the \emph{$k$-th $\cL$-closure} $c_\cL^k(P)$ of $P$ is defined recursively by  $c_\cL^{k}(P)=c_\cL(c_\cL^{k-1}(P))$ for $k \in \N$ and $c_\cL^0(P)=P$. Clearly, $c_\cL^k(P)$ contains the \emph{mixed-integer hull} $P_\MI:= \conv(P \cap (\Z^d \times \R^n))$ of $P$ as a subset. 

The strength of $\cL$-cuts can be formalized by describing how well $P_{\MI}$ is approximated by the sequence $\bigl(c_\cL^k(P)\bigr)_{k=0}^{\infty}$. The strongest form of approximation is expressed as finite convergence to the mixed-integer hull. In the case of finite convergence, $c_\cL^k(P)$ coincides with $P_{\MI}$ if $k$ is large enough (where the choice of $k$ depends on $P$ in general). If the latter holds for every rational polyhedron $P \subseteq \R^{d+n}$, we say that $\cL$ has finite \emph{convergence property}. Note that, a priori, the finite convergence property depends not only on $\cL$ but also on the number of continuous variables $n$. 

As an example, the family of split sets that corresponds to the classical split cuts has finite convergence property for $n=0$ (this follows from
the fact that the well-known Gomory cuts are a subclass of split cuts and one has finite convergence for Gomory cuts; see e.g. \cite[\S 23]{MR874114}) 
but does not have finite convergence property for every $n>0$ (see \cite{MR1059391}). 
Thus, more complicated lattice-free sets have to be used for achieving finite convergence in the mixed-integer setting. Del Pia and Weismantel showed in \cite{MR2968262} that the family of all integral lattice-free polyhedra has finite convergence property for every $n$. Since, for a general $d$, the family of all integral lattice-free polyhedra is complicated and hard to describe explicitly, having a smaller family of lattice-free polyhedra with finite convergence property for every $n$ is desirable. It turns out that every integral lattice-free polyhedron is a subset of a $\Z^d$-maximal integral lattice-free polyhedron. This is stated without proof in \cite{MR2968262} and can be proven using results in \cite{MR2832401} or arguments from \cite{MR2855866}. Consequently, the family of all $\Z^d$-maximal integral lattice-free polyhedra has finite convergence property for every $n$. Furthermore, no proper subfamily of the latter family has finite convergence when $n > 0$, 
as follows from results of \cite{MR2931286}. In view of the above comments, it would be interesting to have an explicit description of all $\Z^d$-maximal integral lattice-free polyhedra. 

In \cite{MR2855866} and \cite{MR2832401} it was shown that, there exist finitely many polyhedra $L_1,\ldots,L_N$ with $N$ depending on $d$ such that every $\Z^d$-maximal integral lattice-free polyhedron $L$ is an image $\phi(L_i)$ of some $L_i$, $i \in \{1,\ldots,N\}$, under an affine transformation $\phi$ satisfying $\phi(\Z^d)=\Z^d$. Affine transformations $\phi$ of $\R^d$ satisfying $\phi(\Z^d) = \Z^d$ are called \emph{unimodular}, while subsets $A, B$ of $\R^d$ satisfying $\phi(A) = B$ for some unimodular transformation $\phi$ are said to be \emph{unimodularly equivalent} (we write $A \simeq B$ for sets $A,B$ which are unimodularly equivalent). 
Thus, the family of $\Z^d$-maximal integral lattice-free polyhedra is finite up to unimodular equivalence. The problem of enumerating this (essentially finite) family for small dimensions was raised in \cite{MR2855866}. For dimensions $d=1$ and $d=2$ enumeration can be carried out easily. 
Already, for dimension $d=3$, this is a challenging problem which was only partially addressed in \cite{MR2855866}.
Using properties of $\R^d$-maximal lattice-free sets, the authors of \cite{MR2855866} were able to enumerate the family of all $\R^3$-maximal integral 
lattice-free polyhedra. While, a priori, this is only a subfamily of the family of $\Z^3$-maximal integral lattice-free polyhedra,
our Theorem~\ref{thm:main_three}, shows that the classification given in \cite{MR2855866} also provides an enumeration of $\Z^3$-maximal integral lattice-free polyhedra.

\begin{corollary}
 Up to unimodular equivalence, there are exactly twelve bounded $\Z^3$-maximal integral lattice-free polyhedra 
(see Figures~\ref{fig:s1d3_width_two} and \ref{fig:s1d3_width_three})
and two unbounded ones, namely $[0,1] \times \R^2$ and $\conv((0,0), (2,0), (0,2)) \times \R$.
\end{corollary}

\begin{figure}[h!]
    \begin{center}
        \begin{tabular}{ccc}
            $M_{4,6}$ &  $M_{4,4}$ & $M_{4,2}$ \\
            \qquad \includegraphics[scale=0.30]{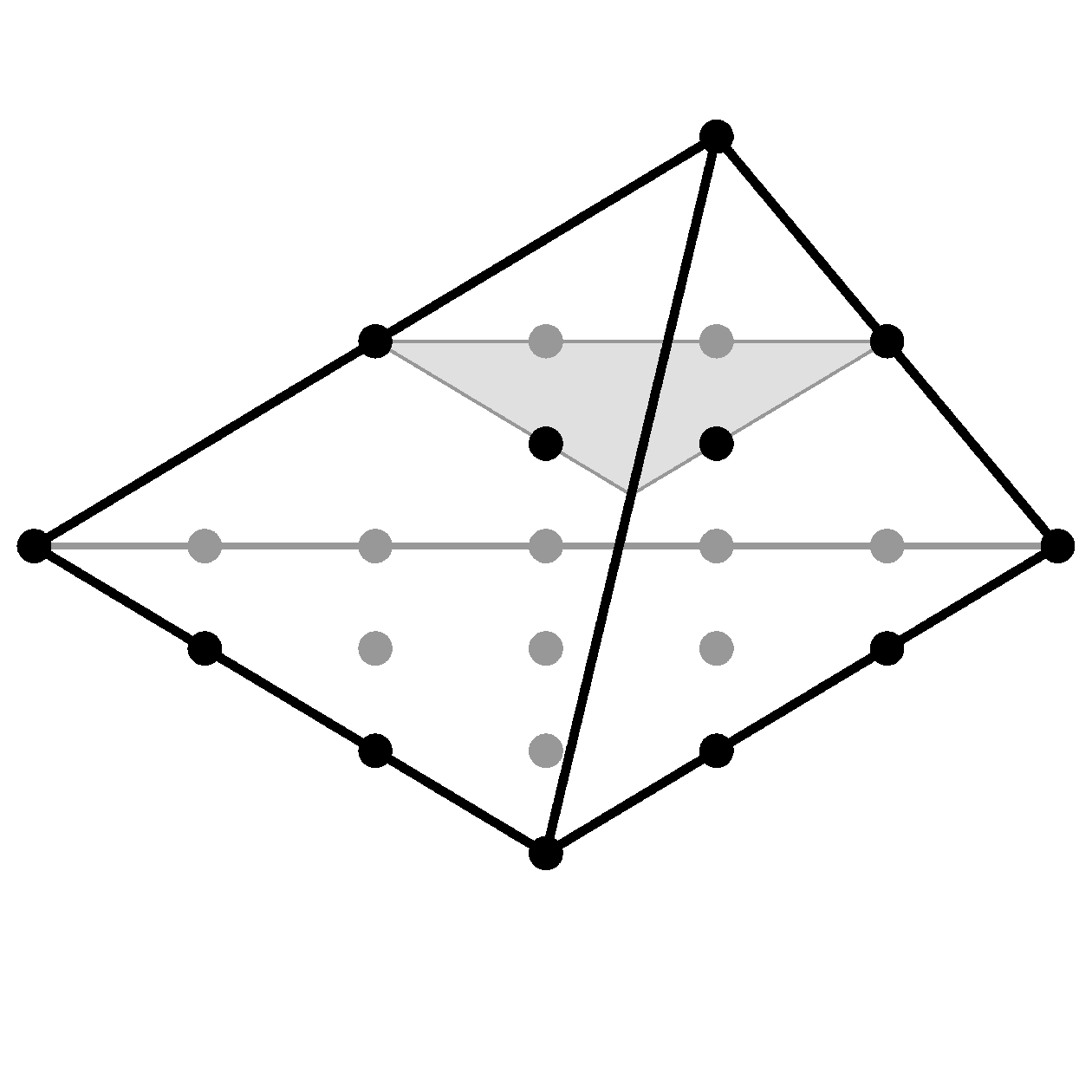} \qquad
            & \qquad \includegraphics[scale=0.30]{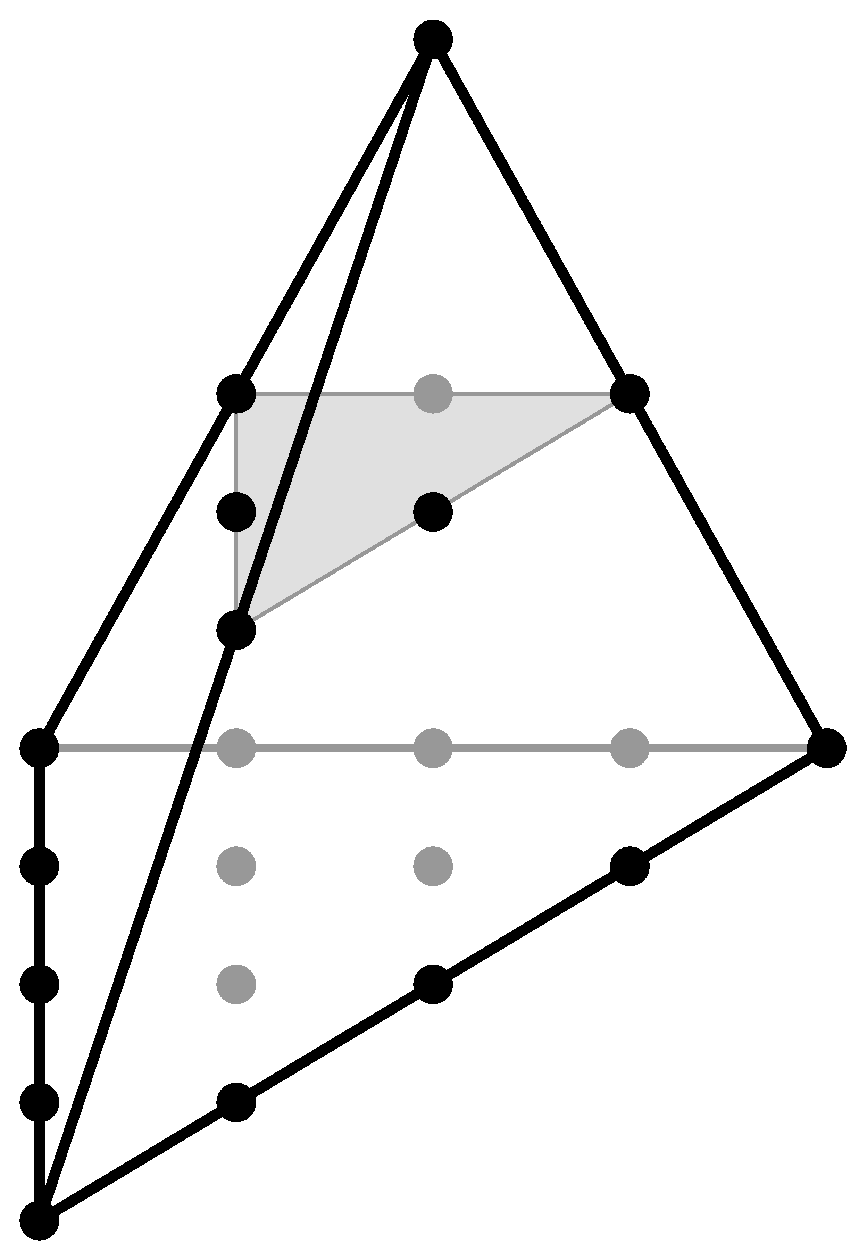} \qquad
            & \qquad \includegraphics[scale=0.30]{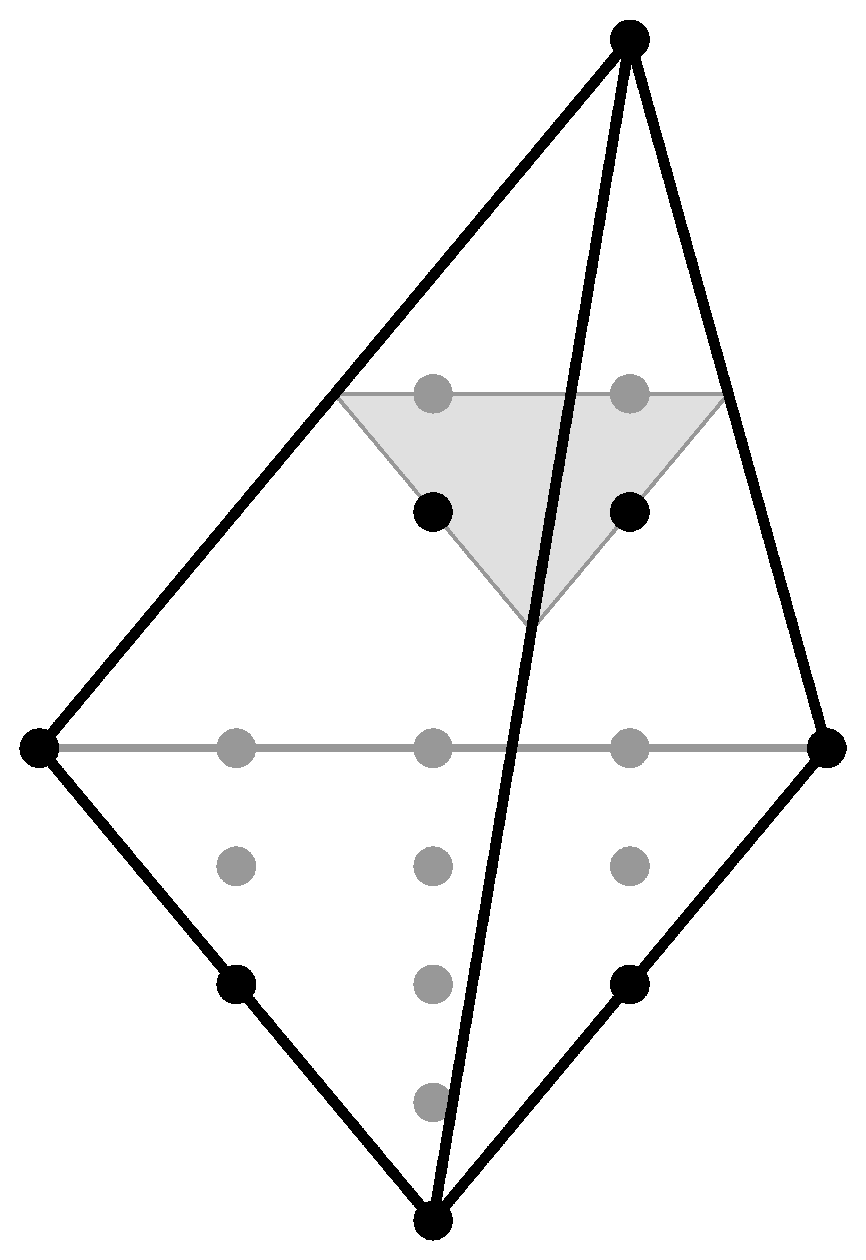} \qquad \\
        \end{tabular}
        \begin{tabular}{cccc}
            $M'_{4,4}$ &  $M_{5,4}$ & $M_{5,2}$ & $M_{6,2}$ \\
            \qquad \includegraphics[scale=0.30]{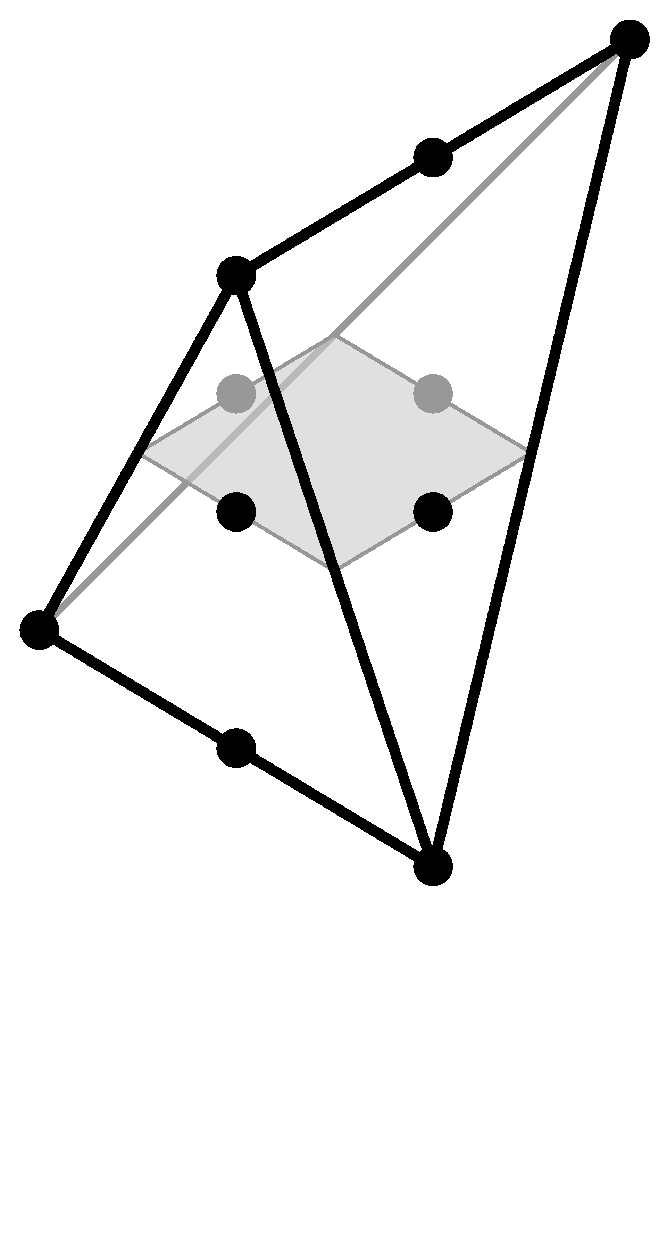} \qquad
            & \qquad \includegraphics[scale=0.30]{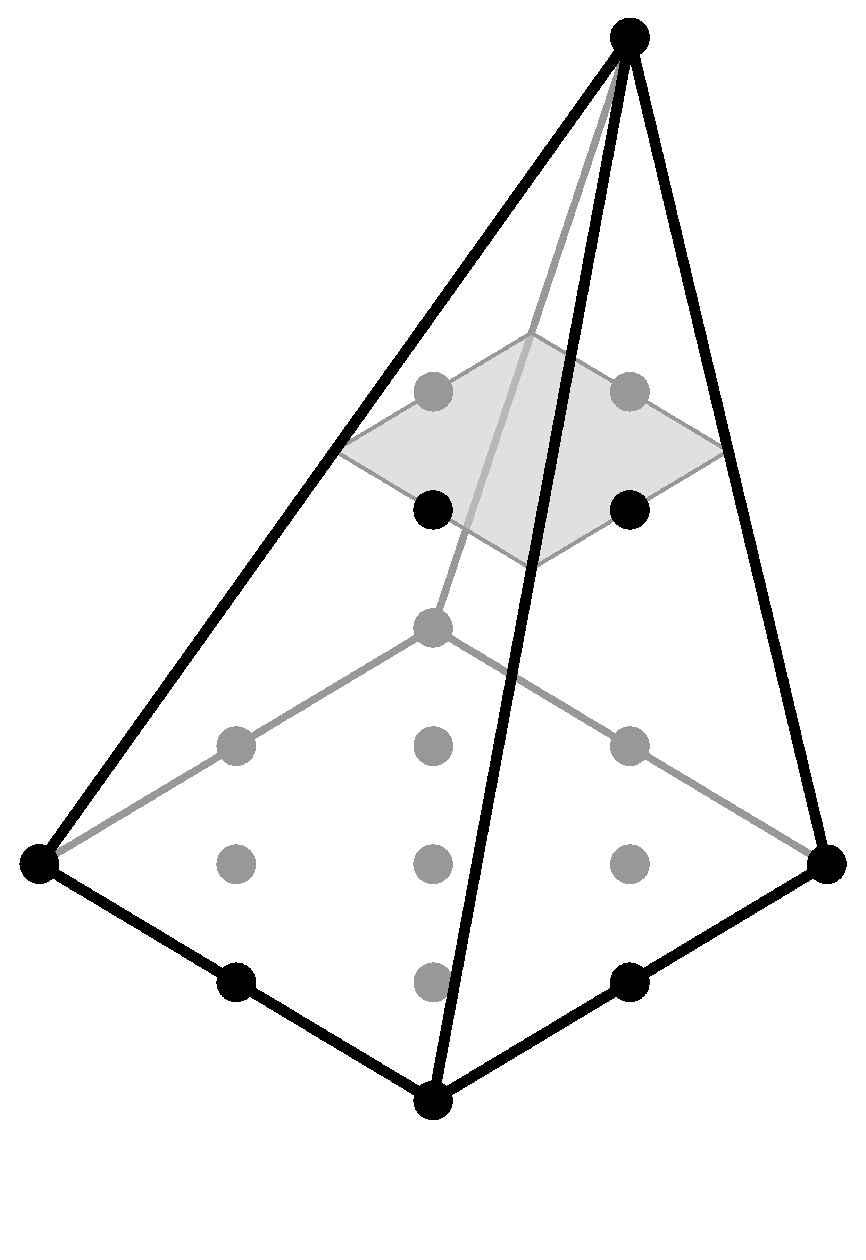} \qquad
            & \qquad \includegraphics[scale=0.30]{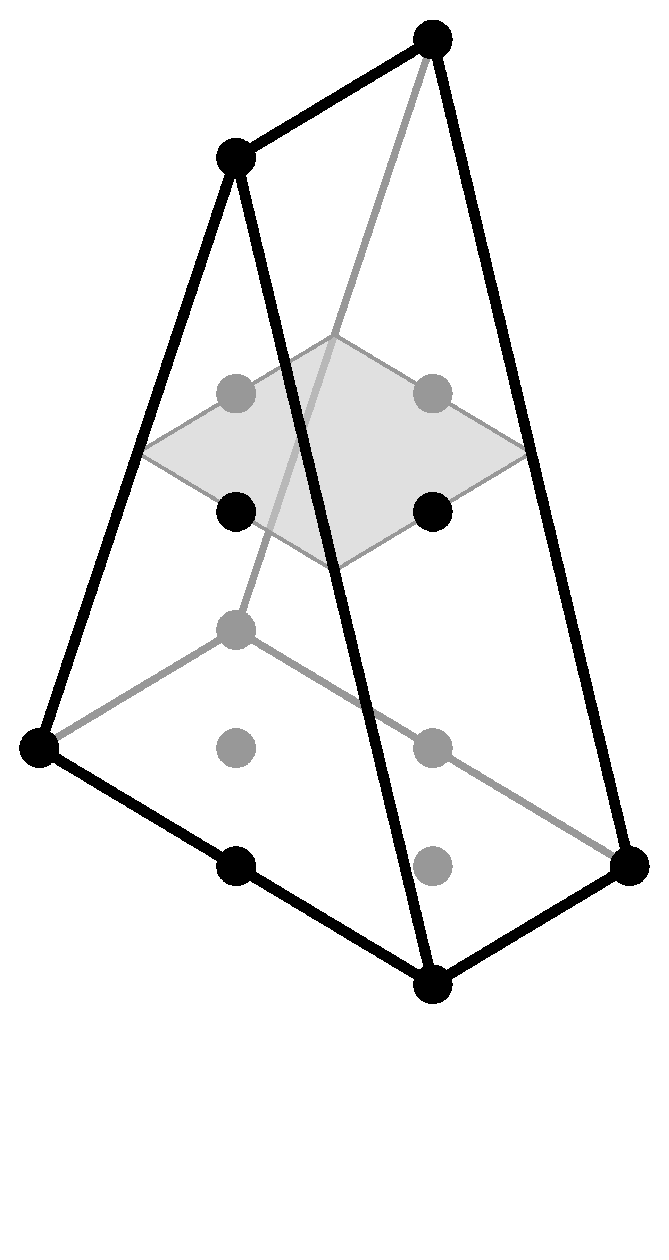} \qquad
            & \qquad \includegraphics[scale=0.30]{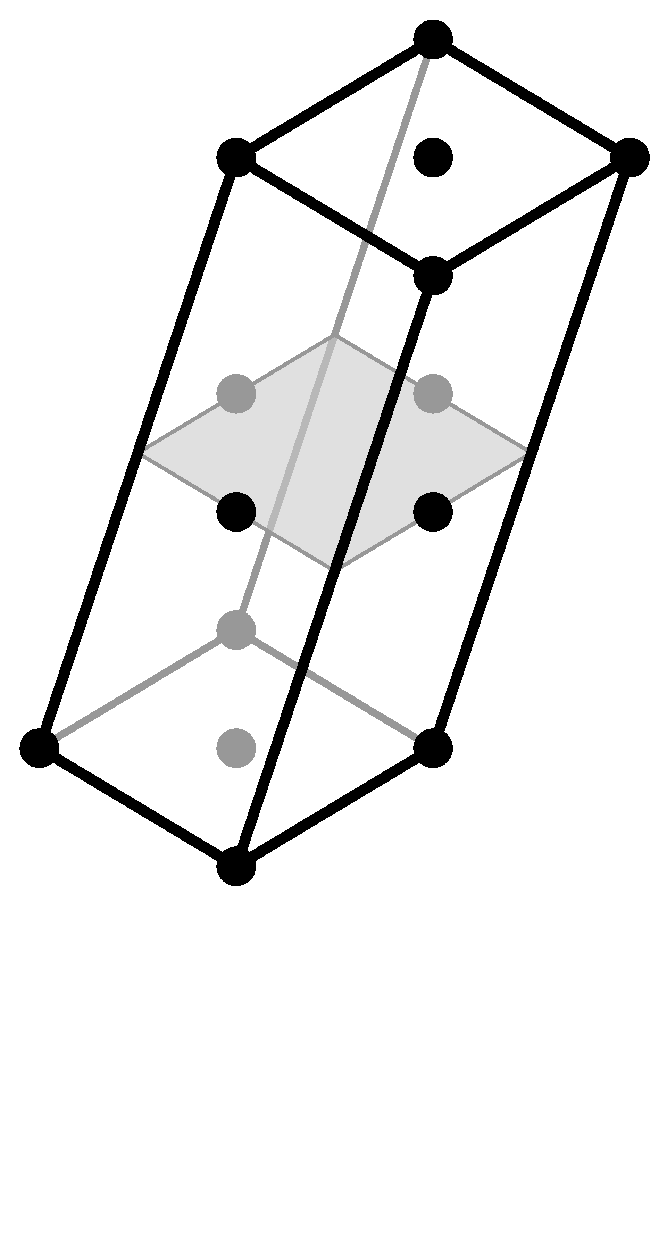} \qquad
        \end{tabular}
    \end{center}
    \caption{The $\Z^3$-maximal integral lattice-free polytopes with lattice width two. For further reference, the
    polytopes are labeled by a pair of indices $(i,j)$, where $i$ is the number of facets and $j$ the lattice diameter
    (defined at the end of the introduction).}
    \label{fig:s1d3_width_two}
\end{figure}

\begin{figure}[h!]
    \begin{center}
        \includegraphics[scale=0.30]{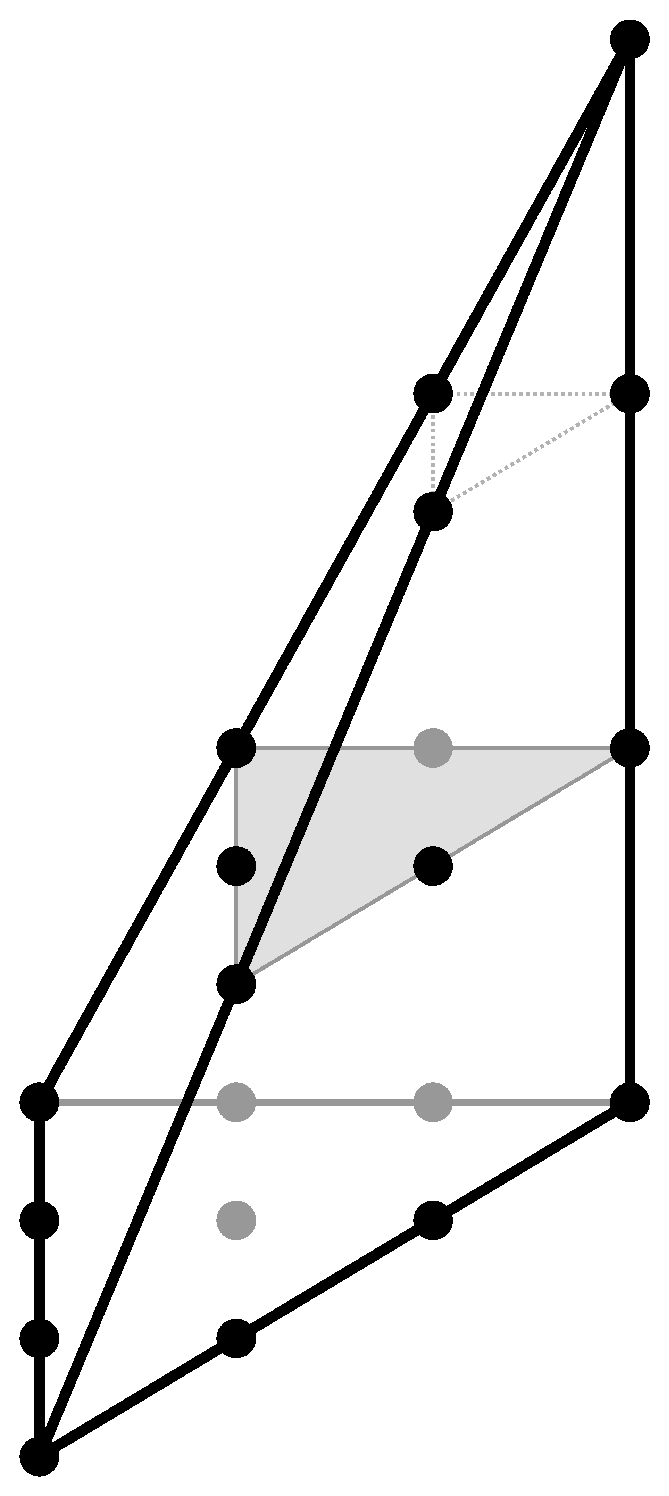}
        \quad \includegraphics[scale=0.30]{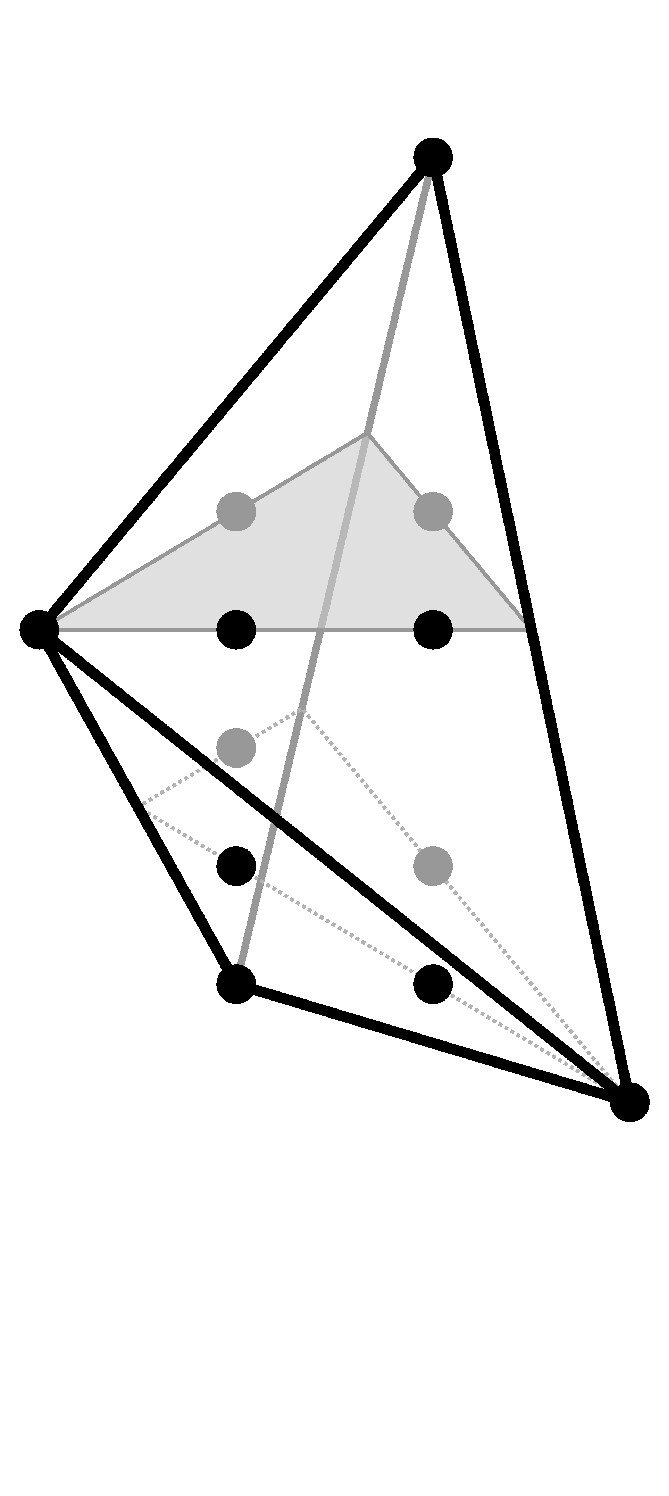}
        \quad \includegraphics[scale=0.30]{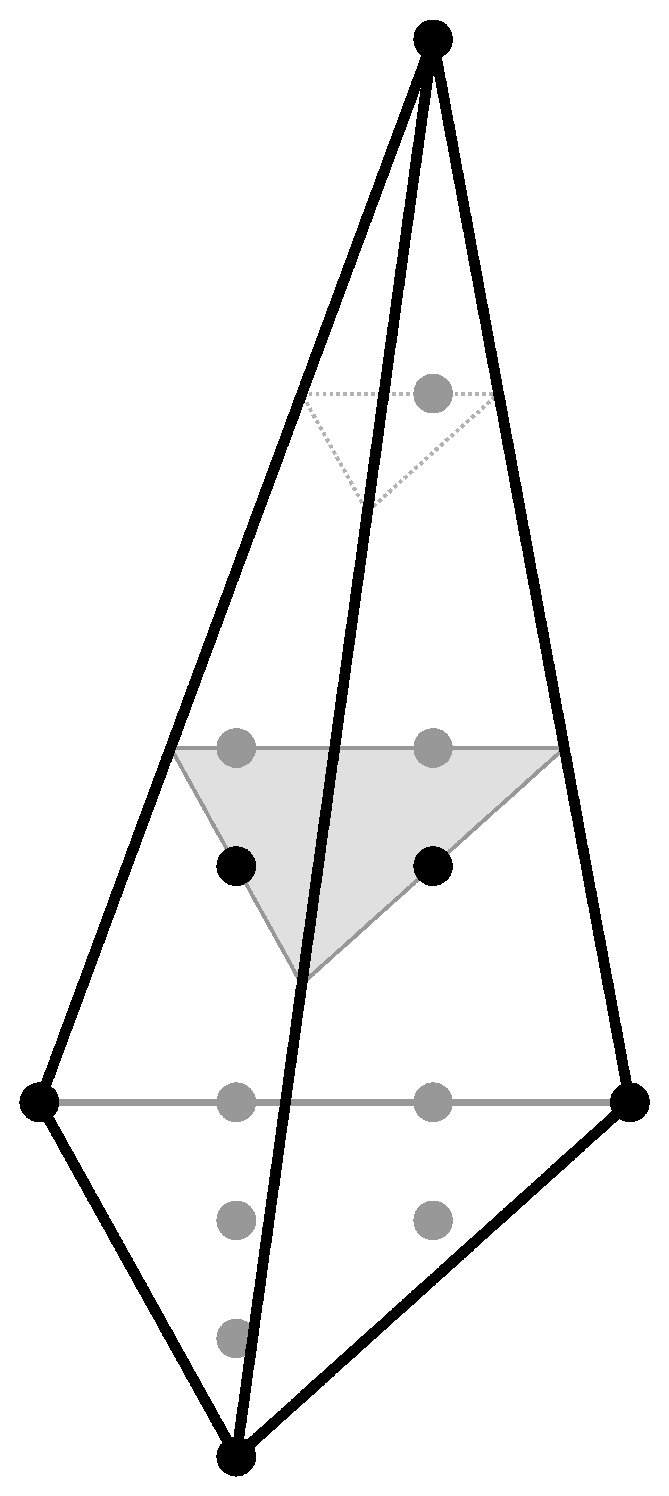}
        \quad \includegraphics[scale=0.30]{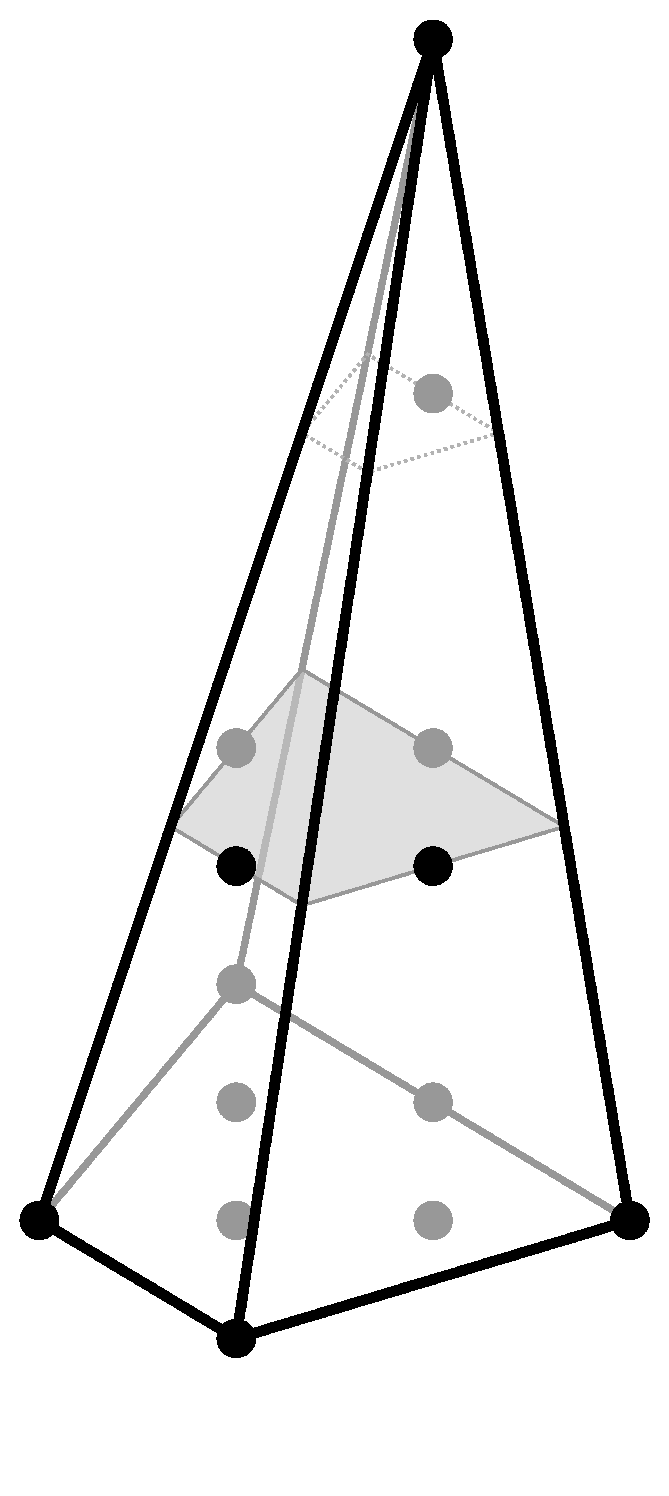}
        \quad \includegraphics[scale=0.30]{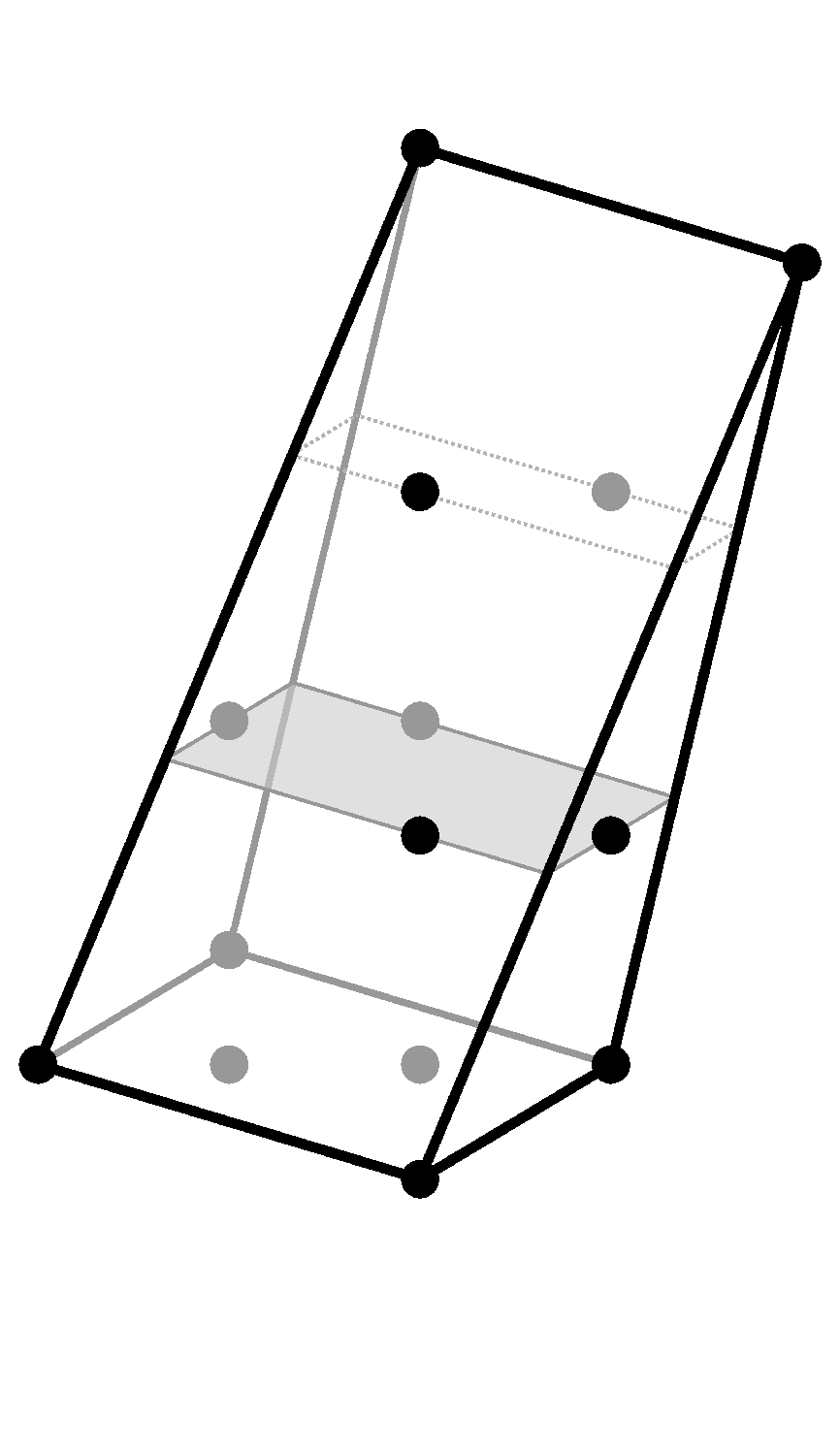}
    \end{center}
    \caption{The $\Z^3$-maximal integral lattice-free polytopes with lattice width three.}
    \label{fig:s1d3_width_three}
\end{figure}

We remark that the main property of $\R^d$-maximal lattice-free polyhedra which was exploited in the classification provided in
\cite{MR2855866} is the following: a lattice-free polyhedron $P$ is $\R^d$-maximal if and only if each of its facets contains an integral point in its
relative interior; see \cite{MR1114315}. We call a facet with the latter property {\em blocked} and we will use this characterization of $\R^d$-maximality repeatedly throughout
this article.
 
\subsection*{Proof strategy}

In the proof of Theorem~\ref{thm:main_three}, we use a classification of all $\Z^2$-maximal polytopes in $\cP(\frac{1}{2}\Z^d)$. This is provided in Section~\ref{sec:s2d2}. Every such polytope is contained in an $\R^2$-maximal lattice-free convex set $L$ in the
plane and its vertices then have to be contained in $L \cap \frac{1}{2}\Z^2$. We give a slightly extended version of the well-known classification of 
$\R^2$-maximal lattice-free convex sets $L$
which allows us to enumerate all $\Z^2$-maximal lattice-free $\frac{1}{2}\Z^2$-polyhedra.

We then turn to integral $\Z^3$-maximal lattice-free polyhedra in dimension three. 
We can restrict ourselves to polytopes here since every unbounded $\Z^3$-maximal lattice-free polyhedron
in $\cP(\Z^3)$ is unimodularly equivalent to $[0,1] \times \R^2$ or $\conv(o, 2e_1,2e_2) \times \R$ in view of \cite[Proposition 3.1]{MR2855866}.
The $\R^3$-maximality of those two sets is easy to see.
Furthermore, one can easily observe that a lattice-free polytope in $\cP(\Z^3)$ with lattice width one (see page~\pageref{page:lw}) cannot be $\Z^3$-maximal, since 
any such polytope is properly contained in an unbounded lattice-free integral polyhedron formed by two adjacent parallel lattice hyperplanes of $\Z^3$. 
Therefore,
it suffices to consider only the polytopes with lattice width at least two, since the lattice width of an integral polytope is integral. 
We distinguish two cases: polytopes with lattice width two and polytopes
with lattice width at least three. 
In the case of lattice width two, 
we first analyze $\Z^d$-maximal lattice-free integral polytopes with this lattice width in general dimension. 
Assuming a $\Z^d$-maximal lattice-free polytope $P$ to be contained in $\R^{d-1} \times [-1,1]$,
we deduce properties of $P_0 = \setcond{x \in \R^{d-1}}{(x,0) \in P}$ and on the faces of $P$ contained in $\R^2 \times \{-1\}$
and $\R^2 \times \{1\}$, respectively. In particular, we obtain that $P_0$ is a $\Z^{d-1}$-maximal
$\frac{1}{2}\Z^{d-1}$-polytope.
Narrowing the focus to the case of dimension three, we then make use of the fact that
$P_0$ has to be one of the four $\Z^2$-maximal lattice-free $\frac{1}{2}\Z^2$-polytopes which we enumerated in Section~\ref{sec:s2d2}. 
Based on that, we classify all $\Z^3$-maximal polytopes in $\cP(\Z^3)$ which have lattice width two. It turns out 
that every such polytope is also $\R^3$-maximal. 

In Section~\ref{sec:lw_three}, we complete the proof of Theorem~\ref{thm:main_three} by a computer search. 
To this end, we 
develop an algorithm which finds all $\Z^3$-maximal polytopes in $\cP(\Z^3)$ with lattice width at least three 
and verifies that all of them are also $\R^3$-maximal.
For enumerating all such polytopes, 
we establish 
bounds on the volume, the first successive minimum and the lattice diameter of $\Z^3$-maximal integral lattice-free polytopes with 
lattice width at least three, making use of classical
results from the geometry of numbers. 
Finally, an implementation of the search algorithm is provided in the appendix.

\subsection*{Basic notation and terminology}
For background information on the theory of mixed-integer optimization, convex sets and polyhedra, and the geometry of
numbers, see for example \cite{MR874114}, \cite{MR1940576}, \cite{MR1451876}, \cite{MR1216521} and \cite{MR893813}.

Throughout the text, $d \in \N$ is the dimension of the ambient space $\R^d$, which is equipped with the standard Euclidian scalar product,
denoted by $\sprod{\cdot}{\cdot}$. By $o$ we denote the zero vector and $e_i$ denotes the $i$-th unit vector, where in both cases 
the dimension is given by the context.
For two points $x,y \in \R^d$, we denote the closed line segment
connecting those points by $[x,y]$.
Accordingly, one has the open line segment $(x,y) := \setcond{\lambda x + (1-\lambda) y}{0 < \lambda < 1}$.
For a subset $X$ of $\R^d$, we write $Y \subseteq X$ if $Y$ is a subset of $X$ and $Y \varsubsetneq X$ if, furthermore, $Y \neq X$ holds.
The cardinality of $X$ is denoted by $|X|$. By $-X$ we denote the set $\setcond{-x}{x \in X}$.
For two subsets $X,Y$ of $\R^d$, we write $X + Y := \setcond{x+y}{x \in X, y \in Y}$ for the {\em Minkowski sum} of $X$ and $Y$, 
while $X - Y$ is short for $X + (-Y)$.
For $t \in \R$, we write $t X := \setcond{tx}{x \in X}$.
For a set $Z \subseteq \Z$, we denote by $\gcd(Z)$ the {\em greatest common divisor} of the set $Z$. 
We call a vector $z \in \Z^d \setminus \{o\}$ {\em primitive} if the greatest common divisor of the components of $z$ is one.

For a set $X \subseteq \R^d$, we use $\conv(X)$, $\aff(K)$ and $\lspan(X)$ to denote the \emph{convex hull}, \emph{affine hull} and \emph{linear hull} of $X$, respectively.
By $\bd{K}$, $\relbd{K}$,
$\relintr{K}$  and $\intr{K}$ we denote the {\em boundary}, the {\em relative boundary}, the \emph{relative interior} and the \emph{interior} of a 
convex set $K$, 
respectively. By $\vol(K)$, we denote the {\em volume} of $K$ and $\dim(K)$ denotes the {\em dimension} of $K$, which is defined as the dimension
of the affine hull of $K$. 
A convex set $K$ such that $K = -K$ is said to be $o$-symmetric. A compact convex set with non-empty interior is called {\em convex body}.
For a convex body $K \subseteq \R^d$, the convex body $K +(- K)$ is called the {\em difference body} of $K$.
For a convex set $K \subseteq \R^d$, we say that $K'$ is a {\em homothetic copy} of $K$
if $K' = \lambda K + v$ for some $\lambda \in \R$ and $v \in \R^d$. We call
$K'$ a positive (non-negative) homothetic copy of $K$ if $\lambda$ is a positive (non-negative) number. 
For a convex set $K \subseteq \R^d$ and a vector $u \in \R^d \setminus \{o\}$, 
we denote by $h(K,u)$ the \emph{support function} $h(K,u) := \operatorname{sup}\setcond{\sprod{u}{x}}{x \in K}$. 

Let $b_1, \ldots, b_d$ linearly independent elements of $\R^d$. 
Then we call the group $\Lambda = \setcond{z_1 b_1 + \ldots + z_d b_d}{z_1, \ldots, z_d \in \Z}$ a {\em lattice} of rank $d$. 
The vectors $b_1, \ldots, b_d$ are said to
form a {\em lattice basis}. The {\em determinant} of $\Lambda$ is defined as the determinant
of the matrix having columns $b_1, \ldots, b_d$.
We say that a subgroup $H$ of a lattice $\Lambda$ is a {\em lattice hyperplane} if $H$ spans a $(d-1)$-dimensional affine subspace of $\R^d$.
Two parallel lattice hyperplanes are said to be {\em adjacent} if there are no points of $\Lambda$ properly between them.

For a subset $X$ of $\R^d$ and a vector $u \in \R^d$, by
\[
 \width{X,u} = \sup_{x \in X} \sprod{u}{x} - \inf_{y \in X} \sprod{u}{y}
\]
we denote the {\em width function} of $X$. If $u$ is a primitive integral vector, then
$\width{X,u}$ is called the {\em (lattice) width} of $X$ in the direction of $u$.
Note that if $X$ is compact (say, a convex body), the infimum and supremum, respectively, are attained.
The {\em lattice width}\label{page:lw} of a set $X \subseteq \R^d$ is defined as the infimum of the width function of $X$ over all non-zero integral vectors, i.e. 
\[
 \lw{X} := \inf_{z \in \Z^d \setminus \{o\}} \width{X,z}.
\]
Again, for a convex body the infimum is attained. We say that the lattice width of a convex body $K$ is {\em attained}
for $z' \in \Z^d \setminus \{o\}$ if $\lw{K} = \width{K,z}$.

A \emph{polyhedron} $P \subseteq \R^d$ is the intersection of finitely many closed halfspaces.
A bounded polyhedron is called a {\em polytope} and a two-dimensional polytope is called {\em polygon}. 
A polyhedron $P$ is called {\em rational} if there exists $s \in \N$ such that $P \in \cP(\frac{1}{s}\Z^d)$, that is, $P = \conv(P \cap \frac{1}{s}\Z^d)$.

For a polyhedron $P$ and a vector $u \in \R^d \setminus \{o\}$, if $h(P,u)$ is finite, we define $F(P,u) := \setcond{x \in P}{\sprod{x}{u} = h(P,u)}$
and call $F(P,u)$ a {\em face} of $P$.
If $F := F(P,u)$ has dimension $\dim(P) - 1$, we say that $F$ is a {\em facet} of $P$ and say that 
$u$ is an {\em outer normal (facet) vector} of $F$. If $F$ consists only of a single point, then this point is called a {\em vertex} of $P$.
The set of vertices of $P$ is denoted by $\vertset{P}$. 
For a $d$-dimensional rational polyhedron $P$, we define $U(P)$ to be the set of all primitive vectors which are outer normal vectors of facets of $P$. 
Note that if $P$ is rational,
every facet of $P$ can be given as $F(P,u)$ for some $u \in U(P)$. 
Let $x \in \bd{P}$. Then the {\em normal cone} of $P$ at $x$ is defined as $N(P,x) := \{o\} \cup \setcond{u \in \R^d \setminus \{o\}}{x \in F(P,u)}$.
We call a vertex $v$ of $P$ {\em unimodular} in this article if $v$ has the 
property that there are exactly $d$ facets of $P$ containing $v$ and these facets can be given as $F(P,u_1), \ldots, F(P,u_d)$, where
$u_1, \ldots, u_d$ are elements of $U(P)$ forming a basis of the lattice $\Z^d$.

A $d$-dimensional polytope with $d+1$ vertices is called \emph{simplex} of dimension $d$.
Given a $d$-dimensional simplex $S \subseteq \R^d$ with vertices $v_1, \ldots, v_{d+1}$ and a point $x \in \R^d$, the \emph{barycentric coordinates} of 
$x$ with respect to $S$ are uniquely determined real numbers $\beta_1, \ldots, \beta_{d+1}$ satisfying 
$x = \sum_{i=1}^{d+1} \beta_i v_i$ and $\sum_{i=1}^{d+1} \beta_i =1$. In this case, for $i \in \{1, \ldots, d+1\}$, we say that $\beta_i$ is the barycentric coordinate associated with $v_i$.
One has $x \in \intr{S}$ if and only if $\beta_i > 0$ for all $i \in \{1, \ldots, d+1\}$.

\section{Half-integral lattice-free polyhedra in dimension two}\label{sec:s2d2}

As explained in the outline of the proof, in this section we show that for two-dimensional \emph{half-integral} lattice-free
polygons, $\Z^2$-maximality implies $\R^2$-maximality.

We make use of the following characterization of $\R^2$-maximal lattice-free convex sets in $\R^2$.
Although the characterization does not seem to be available in the literature in precisely this form, it is largely known from \cite{MR1060014},
\cite[Proposition 1]{MR2679986} and \cite{MR2890359}.

\begin{theorem}[Classification of maximal lattice-free sets in dimension two]\label{thm:dey_wolsey}
			Let $L$ be an $\R^2$-maximal lattice-free convex set in $\R^2$. 
			Then there exists a unimodular transformation $\varphi$ such that $\varphi(L) \in \cL_1 \cup \ldots \cup \cL_5$, where $\cL_1, \ldots, \cL_5$ are the following families (see Fig.~\ref{fig:mlf_polygons}).
			\begin{enumerate}[label=(\alph*)]
				\item\label{eq:split} $\cL_1$ consists of exactly one polyhedron, namely the unbounded set $[0,1] \times \R$.
				\item\label{eq:axis-aligned} $\cL_2$ is the set of all `axis-aligned' triangles with vertices $v_1,v_2,v_3$ such that $v_1$ coincides with the vertex $o$ of the square $[0,1]^2$, 
				the vertices $v_2,v_3$ satisfy
				\begin{align*}
					v_2 & \in (1,+\infty) \times \{0\}, 
					& v_3 & \in \{0\} \times (1,+\infty) 
				\end{align*}
				and for the vertices $e_1, e_2, e_1 + e_2$ of the square $[0,1]^2$ one has
				\begin{align*}
					e_1 & \in (v_1,v_2), & e_2 & \in (v_1,v_3), &  e_1 + e_2 & \in (v_2,v_3).
				\end{align*}
				\item\label{eq:type2} $\cL_3$ is the set of triangles with vertices $v_1, v_2, v_3$ satisfying
				\begin{align*}
					v_1 & \in (-\infty,0) \times \{0\}, 
					& v_2 & \in (1,+\infty) \times \{0\},
					& v_3 & \in (0,1) \times (1,+\infty)
				\end{align*}
				and such that for the vertices $o, e_1, e_2, e_1 + e_2$ of the square $[0,1]^2$ one has
				\begin{align*}
					o,e_1 & \in (v_1,v_2), & e_2 & \in (v_1,v_3), &  e_1 + e_2 & \in (v_2,v_3).
				\end{align*}
				\item\label{eq:type3} $\cL_4$ is the set of triangles with vertices $v_1, v_2, v_3$ satisfying
				\begin{align*}
					 v_1 & \in \setcond{(x_1,x_2) \in \R^2}{x_2 > 1, \, 0 < x_1 + x_2 < 1},\\
					 v_2 & \in \setcond{(x_1,x_2) \in \R^2}{0 < x_1 < 1, \, x_1 + x_2 < 0},\\
					 v_3 & \in \setcond{(x_1,x_2) \in \R^2}{x_1 > 1, \, 0 < x_2 < 1}
				\end{align*}
				and such that for the vertices $o, e_1, e_2$ of the triangle $\conv(o,e_1,e_2)$ one has
				\begin{align*}
					o & \in (v_1,v_2), & e_1 & \in (v_2,v_3), & e_2 & \in (v_3,v_1).
				\end{align*}
				\item\label{eq:quadr} $\cL_5$ is the set of quadrilaterals with vertices $v_1, v_2, v_3, v_4$ satisfying
				\begin{align*}
					v_1 & \in (-\infty,0) \times (0,1), & v_2 & \in (0,1) \times (-\infty,0),\\
					v_3 & \in (1,+\infty) \times (0,1), & v_4 & \in (0,1) \times (1,+\infty), \\
				\end{align*}
				and such that for the vertices $o, e_1, e_2, e_1 + e_2$ of the square $[0,1]^2$ one has
				\begin{align*}
					o & \in (v_1,v_2), & e_1 & \in (v_2,v_3), & e_1+e_2 & \in (v_3,v_4), & e_2 & \in (v_4,v_1).
				\end{align*}
			\end{enumerate}
		\end{theorem}

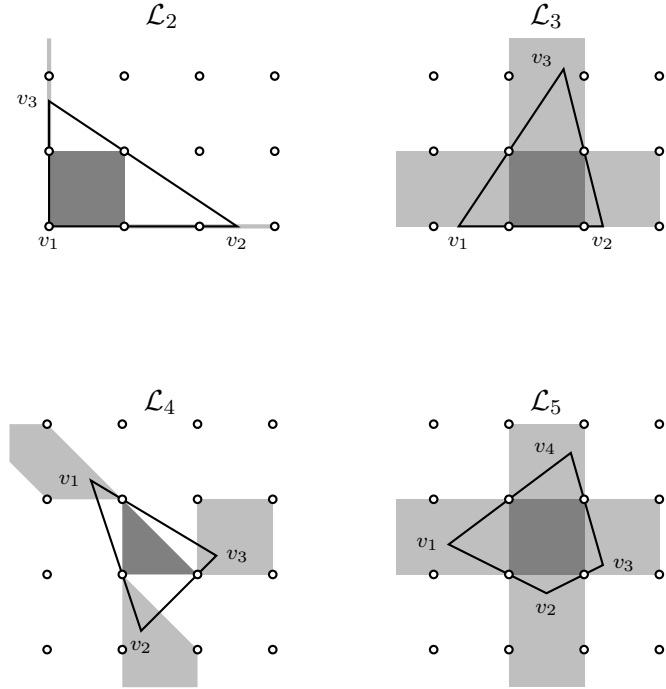
\begin{figure}[h!]
\begin{center}
\begin{tabular}{cc}
\begin{tikzpicture}
\draw[ultra thick, lightgray] (0,1) -- (0,2.5) (1,0) -- (3,0);
\fill[gray] (0,0) -- (0,1) -- (1,1) -- (1,0) --cycle;
 		\draw[fill=black] (0,0) circle (0.05);
 		\draw[fill=black] (0,1) circle (0.05);
 		\draw[fill=black] (1,0) circle (0.05);
 		\draw[fill=black] (1,1) circle (0.05);
		\draw[fill=black] (2,0) circle (0.05);
 		\draw[fill=black] (0,2) circle (0.05);
 		\draw[thick] (0,0) -- (2.5,0) -- (0,5/3) --cycle;
 		\node[above] at (1.5,2.5) {$\cL_2$};
\node[left] at (0,5/3) {\scriptsize $v_3$};
\node[below] at (0,0) {\scriptsize $v_1$};
\node[below] at (2.5,0) {\scriptsize $v_2$};
             \foreach \x in {0,...,3}
            {
                \foreach \y in {0,...,2}
                {
                    \fill[white,draw=black, thick] (\x,\y) circle (0.5mm);
                }
            }
\end{tikzpicture}
& \hspace{1cm}
\begin{tikzpicture}
\fill[lightgray] (0,1) -- (1,1) -- (1,2.5) -- (0,2.5) --cycle;
\fill[lightgray] (0,0) -- (-1.5,0) -- (-1.5,1) -- (0,1) --cycle;
\fill[lightgray] (1,0) -- (2,0) -- (2,1) -- (1,1) --cycle;
\fill[gray] (0,0) -- (0,1) -- (1,1) -- (1,0) --cycle;
 		\draw[fill=black] (0,0) circle (0.05);
 		\draw[fill=black] (0,1) circle (0.05);
 		\draw[fill=black] (1,0) circle (0.05);
 		\draw[fill=black] (1,1) circle (0.05);
		\draw[thick] (-2/3,0) -- (5/4,0) -- (8/11,23/11) --cycle;
 		\node[above] at (0.5,2.5) {$\cL_3$};
\node[below] at (-2/3,0) {\scriptsize $v_1$};
\node[below] at (5/4,0) {\scriptsize $v_2$};
\node[above left] at (8/11,2) {\scriptsize $v_3$};
             \foreach \x in {-1,...,2}
            {
                \foreach \y in {0,...,2}
                {
                    \fill[white,draw=black, thick] (\x,\y) circle (0.5mm);
                }
            }
\end{tikzpicture}
\\ \vspace{1cm} & \\
\begin{tikzpicture}
            \fill[lightgray] (0,0) -- (1,-1) -- (1,-3/2) -- (0,-3/2) -- cycle;
            \fill[lightgray] (1,0) -- (1,1) -- (2,1) -- (2,0) -- cycle;
            \fill[lightgray] (0,1) -- (-1,1) -- (-3/2,3/2) -- (-3/2,2) -- (-1,2) -- cycle;
\fill[gray] (0,0) -- (0,1) -- (1,0) --cycle;
 		\draw[fill=black] (0,0) circle (0.05);
 		\draw[fill=black] (0,1) circle (0.05);
 		\draw[fill=black] (1,0) circle (0.05);
		\draw[thick] (5/4,1/4) -- (1/4,-3/4) -- (-5/12,5/4) --cycle;
 		\node[above] at (0.5,2) {$\cL_4$};
\node[left] at (-5/12,5/4) {\scriptsize $v_1$};
\node[below] at (1/4,-3/4) {\scriptsize $v_2$};
\node[right] at (5/4,1/4) {\scriptsize $v_3$};
             \foreach \x in {-1,...,2}
            {
                \foreach \y in {-1,...,2}
                {
                    \fill[white,draw=black, thick] (\x,\y) circle (0.5mm);
                }
            }
\end{tikzpicture}
& \hspace{1cm}
\begin{tikzpicture}
\fill[lightgray] (0,1) -- (1,1) -- (1,2) -- (0,2) --cycle;

\fill[lightgray] (0,0) -- (-3/2,0) -- (-3/2,1) -- (0,1) --cycle;

\fill[lightgray] (1,0) -- (2,0) -- (2,1) -- (1,1) --cycle;

\fill[lightgray] (0,0) -- (0,-3/2) -- (1,-3/2) -- (1,0) --cycle;
\fill[gray] (0,0) -- (0,1) -- (1,1) -- (1,0) --cycle;
 		\draw[fill=black] (0,0) circle (0.05);
 		\draw[fill=black] (0,1) circle (0.05);
 		\draw[fill=black] (1,0) circle (0.05);
 		\draw[fill=black] (1,1) circle (0.05);
		\draw[thick] (-0.8,0.4) -- (0.5,-0.25) -- (1.25,0.125) -- (14/17,55/34) --cycle;
 		\node[above] at (0.5,2) {$\cL_5$};
\node[left] at (-0.8,0.4) {\scriptsize $v_1$};
\node[below] at (0.5,-0.25) {\scriptsize $v_2$};
\node[right] at (1.25,0.125) {\scriptsize $v_3$};
\node[left] at (13/17,57/34) {\scriptsize $v_4$};
             \foreach \x in {-1,...,2}
            {
                \foreach \y in {-1,...,2}
                {
                    \fill[white,draw=black, thick] (\x,\y) circle (0.5mm);
                }
            }
\end{tikzpicture}
\end{tabular}
\caption{Illustration of the classes $\cL_2, \ldots, \cL_5$ of Theorem~\ref{thm:dey_wolsey}. 
The area shaded in darker grey is $[0,1]^2$ for $\cL_2$, $\cL_3$, $\cL_5$ and $\conv(o, e_1, e_2)$ for $\cL_4$. The areas shaded in light grey are
the regions containing the vertices of $\varphi(L)$.}\label{fig:mlf_polygons}
\end{center}
\end{figure}

\begin{proof}
If $L$ is unbounded, then it is well known from \cite[\S 3]{MR1114315}, 
it is unimodularly equivalent to $[0,1] \times \R \in \cL_1$.
If $L$ is bounded, from \cite[\S 3]{MR1114315} we have that $L$ is either a triangle or a quadrilateral and again from \cite[\S 3]{MR1114315} we have that
each edge of $L$ is blocked. 
It is well-known (see, e.g., \cite{MR1060014},
\cite[Proposition~1]{MR2679986}) that if $L$ is a quadrilateral, then the four integral points blocking its edges can be unimodularly mapped onto $\{0,1\}^2$.
It is then easy to see that this mapping transforms $L$ into an element of $\cL_5$. 
This leaves the case that $L$ is a triangle. Assume first that
$L$ has an integral vertex $v$. 
Then we can choose an integral point in the relative interior of each of the three edges of $L$ such that those three points together with $v$
can be unimodularly mapped onto $\{0,1\}^2$ (see again \cite[Proposition~1]{MR2679986}) where $v$ is mapped onto $o$. Thus, $\varphi(L) \in \cL_2$.
We now switch to the case that $L$ does not have integral vertices.
If $L$ 
contains at least four integral points, then the relative interior of some edge of $L$ contains at least two integral points $u_1,u_2$ such that $(u_1,u_2) \cap \Z^2 = \emptyset$. 
Together with one point from the relative interior of each of the two remaining edges, they form an integral quadrilateral not containing any other
integral points. Unimodularly mapping this quadrilateral onto $\{0,1\}^2$ such that $u_1$ and $u_2$ are mapped onto $o$ and $e_1$, respectively, yields an element of $\cL_3$.

This leaves the case that $L$ has precisely three integral points.
There exists a unimodular
transformation $\varphi$ such that $\varphi(L) \cap \Z^2 = \{o, e_1, e_2\}$ and $\varphi(L)$ has vertices $v_1, v_2, v_3$
with $o  \in (v_1,v_2)$, $e_1  \in (v_2,v_3)$ and $e_2  \in (v_3,v_1)$ (see, e.g., \cite{MR1060014}). 
We prove $\varphi(L) \in \cL_4$, for which it remains to show that $v_1,v_2,v_3$ are situated as claimed
in \ref{eq:type3}. We can express $o,e_1,e_2$ with respect to $v_1,v_2,v_3$ as follows: 
there exist $0 < \alpha_1, \alpha_2, \alpha_3 < 1$ such that 
\begin{align}
o &= \alpha_1 v_1 + (1-\alpha_1) v_2, \notag \\
e_1 &= \alpha_2 v_2 + (1-\alpha_2)v_3, \label{eq:bc_def}\\
e_2 &= \alpha_3 v_3 + (1-\alpha_3)v_1 \notag.
\end{align}
We now want to express $v_1,v_2,v_3$ with respect to $o,e_1,e_2$. We follow the argumentation in \cite[Proof of Lemma 4.2]{MR2890359} and \cite[(7)]{MR1060014} 
and translate \eqref{eq:bc_def} into the following matrix representation:
\begin{align*}
\begin{pmatrix} 1 & 1 & 1 \\ 0 & 1 & 0 \\ 0 & 0 & 1 \end{pmatrix} = \begin{pmatrix} 1 & 1 & 1 \\ v_1 & v_2 & v_3 \end{pmatrix} \begin{pmatrix}
\alpha_1 & 0 & 1 - \alpha_3 \\ 1 - \alpha_1 & \alpha_2 & 0 \\ 0 & 1 - \alpha_2 & \alpha_3 \end{pmatrix}.
\end{align*}
The determinant of the matrix involving $\alpha_1, \alpha_2, \alpha_3$ is $D : = \alpha_1 \alpha_2 \alpha_3 + (1-\alpha_1)(1 - \alpha_2)(1 - \alpha_3) > 0$. Hence,
this matrix is invertible and we can
 directly compute the components of $v_1,v_2,v_3$ via Cramer's rule:
\begin{align}
v_1 &= \frac{1}{D}\Big(-(1-\alpha_1)\alpha_3, \; (1-\alpha_1)(1-\alpha_2)\Big), \notag \\
v_2 &= \frac{1}{D}\Big(\alpha_1\alpha_3, \; -\alpha_1(1-\alpha_2)\Big), \label{eq:bc_ineq_2}\\
v_3 &= \frac{1}{D}\Big((1-\alpha_1)(1-\alpha_3), \; \alpha_1\alpha_2\Big). \notag
\end{align}
In \cite[Lemma 5.1]{MR2890359} it was shown that either 
\begin{align}
 \alpha_i + \alpha_j < 1 \text{ for all } 1 \le i < j \le 3 \label{eq:bc_ineq_1}
\end{align}
or $\alpha_i + \alpha_j > 1 \text{ for all } 1 \le i < j \le 3$. 
In the case that \eqref{eq:bc_ineq_1} holds, it is straightforward to verify that the vertices $v_1,v_2,v_3$ are situated as claimed.
If $\alpha_i + \alpha_j > 1 \text{ for every } i,j \in \{1,2,3\}$ holds instead of \eqref{eq:bc_ineq_1}, one can instead prove in a straightforward way
that this yields the following regions for the vertices $v_1,v_2,v_3$:
				\begin{align*}
					 v_1 & \in \setcond{(x_1,x_2) \in \R^2}{x_1 + x_2 < 0, \, 0 < x_2 < 1},\\
					 v_2 & \in \setcond{(x_1,x_2) \in \R^2}{x_1 > 1, \, 0 < x_1 + x_2 < 1}\\
					 v_3 & \in \setcond{(x_1,x_2) \in \R^2}{0 < x_1 < 1, \, x_2 > 1}.
				\end{align*}
Applying the unimodular transformation $(x_1,x_2) \mapsto (x_2,x_1)$,
and appropriately re-indexing the vertices yields a triangle in $\cL_4$. This is again straightforward to verify.
\end{proof}

Based on the previous theorem, we now enumerate all $\Z^2$-maximal lattice-free half-integral polygons up to unimodular equivalence.
The polygons $Q_2, Q_3, Q_4, Q_5$ of the following theorem are depicted in Figure~\ref{fig:s2d2}.

\begin{figure}[h!]
\begin{center}
\begin{tabular}{cc}
	\begin{tikzpicture}
\fill[lightgray] (0,0) -- (2,0) -- (0,2) --cycle;
		\draw[thick](0,0) -- (2,0) -- (0,2) --cycle;
\fill[lightgray] (3,0) -- (6,0) -- (4.5,1.5) --cycle;
		\draw[thick](3,0) -- (6,0) -- (4.5,1.5) --cycle;
		\fill[black] (9/2,3/2) circle (0.5mm);
\fill[lightgray] (7.5,0) -- (9.5,0) -- (8.5,2) --cycle;
		\draw[thick] (7.5,0) -- (9.5,0) -- (8.5,2) --cycle;
		\fill[black] (7.5,0) circle (0.5mm);
		\fill[black] (8.5,2) circle (0.5mm);
		\fill[black] (9.5,0) circle (0.5mm);
\fill[lightgray] (10.5,0.5) -- (11.5,-0.5) -- (12.5,0.5) -- (11.5,1.5) --cycle;
		\draw[thick] (10.5,0.5) -- (11.5,-0.5) -- (12.5,0.5) -- (11.5,1.5) --cycle;
		\fill[black] (10.5,0.5) circle (0.5mm);
		\fill[black] (11.5,-0.5) circle (0.5mm);
		\fill[black] (12.5,0.5) circle (0.5mm);
		\fill[black] (11.5,1.5) circle (0.5mm);
             \draw[step=0.5, black, dotted] (0,-0.5) grid (13,2);
            \foreach \x in {0,...,13}
            {
                \foreach \y in {0,...,2}
                {
                    \fill[white,draw=black, thick] (\x,\y) circle (0.5mm);
                }
            }
	\end{tikzpicture}
\end{tabular}
\caption{All $\Z^2$-maximal lattice-free $\frac{1}{2}\Z^2$-polytopes.}\label{fig:s2d2}
\end{center}
\end{figure}
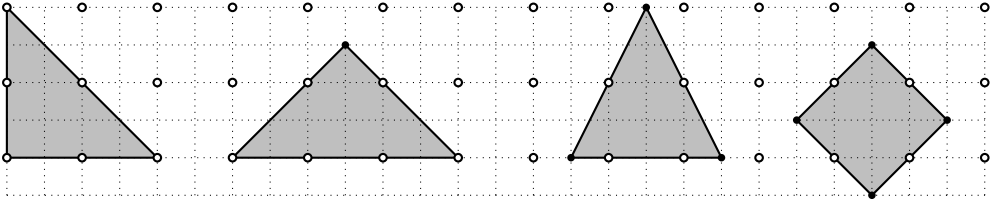

\begin{theorem}[$\Z^2$-maximal lattice-free polyhedra in $\cP(\frac{1}{2}\Z^2)$]\label{thm:half-integral}
Let $Q \in \cP(\frac{1}{2}\Z^2)$. Then $Q$ is $\Z^2$-maximal lattice-free if and only if $Q \simeq Q_i$ for some $i \in \{1, \ldots, 5\}$, where
    \begin{itemize}
\item[] $Q_1$ is the unbounded polyhedron $[0,1] \times \R$,
\item[] $Q_2$ is the triangle with vertices $(0,0),(2,0),(0,2)$,
\item[] $Q_3$ is the triangle with vertices $(\frac{1}{2}\pm \frac{3}{2},0),(\frac{1}{2},\frac{3}{2})$,
\item[] $Q_4$ is the triangle with vertices $(\frac{1}{2} \pm 1,0),(\frac{1}{2},2)$,
\item[] $Q_5$ is the quadrilateral with vertices $(-\frac{1}{2},\frac{1}{2}), (\frac{1}{2},\frac{3}{2}), (\frac{3}{2},\frac{1}{2}), (\frac{1}{2},-\frac{1}{2})$.
    \end{itemize}
In particular, $Q$ is $\Z^2$-maximal lattice-free if and only if $Q$ is $\R^2$-maximal lattice-free.
\end{theorem}

\begin{proof}
Let $Q \in \cP(\frac{1}{2}\Z^2)$ be $\Z^2$-maximal lattice-free. We show that then, $Q \simeq Q_i$ for some $i \in \{1, \ldots, 5\}$.
It is known, see for example \cite[Proposition 3.1]{MR2890359}, that every lattice-free set in $\R^2$ is a subset of an $\R^2$-maximal lattice-free set.
Thus, there exists an $\R^2$-maximal lattice-free set $L$ such that $Q \subseteq L$.
By applying a suitable unimodular transformation, we can assume $L \in \cL_i$ with $i \in \{1, \ldots, 5\}$, where $\cL_1,\ldots, \cL_5$ are defined as in 
Theorem~\ref{thm:dey_wolsey}. We observe that one has
\begin{align}
 \conv(L \cap \Z^2) \subseteq Q \label{eq:q_l_integral}
\end{align}
and
\begin{align}
Q \subseteq \conv\Big(L \cap \frac{1}{2}\Z^2\Big), \label{eq:q_l_halfintegral}
\end{align}
where \eqref{eq:q_l_integral} holds since $Q$ is $\Z^2$-maximal and \eqref{eq:q_l_halfintegral} holds since $Q \in \cP(\frac{1}{2}\Z^2)$ and $Q \subseteq L$. 

{\em Case 1: $L \in \cL_1$.}
Then $L = [0,1] \times \R$ and hence, \eqref{eq:q_l_integral} together with \eqref{eq:q_l_halfintegral} yields $Q = L$.

\begin{figure}[h!]
\begin{center}
\begin{tabular}{ccc}
 \begin{tikzpicture}[scale=1]
\fill[lightgray] (0,1) -- (1,1) -- (1,3) -- (0,3) --cycle;

\fill[lightgray] (1,0) -- (1,1) -- (4,1) -- (4,0) --cycle;

\fill[gray] (0,0) -- (1,0) -- (1,1) -- (0,1) --cycle;

             \draw[step=0.5, black, dotted] (0,0) grid (4,3);
\fill[black] (0,1.5) circle (0.3mm) (0,2) circle (0.3mm)
(0.5,1.5) circle (0.3mm) (0.5,2) circle (0.3mm) (1.5,0.5) circle (0.3mm) (2,0.5) circle (0.3mm)
(2.5,0.5) circle (0.3mm) (3,0.5) circle (0.3mm)
(3.5,0.5) circle (0.3mm) (4,0.5) circle (0.3mm)
(3.5,0) circle (0.3mm) (0,2.5) circle (0.3mm) (0.5,2.5) circle (0.3mm) (0.5,3) circle (0.3mm)
(1.5,0) circle (0.3mm) (2,0) circle (0.3mm) (2.5,0) circle (0.3mm);
            \foreach \x in {0,...,4}
            {
                \foreach \y in {0,...,3}
                {
                    \fill[white,draw=black, thick] (\x,\y) circle (0.5mm);
                }
            }
  \end{tikzpicture}
&
\hspace{1cm}
&
  \begin{tikzpicture}[scale=1]
\fill[lightgray] (0,1) -- (1,1) -- (1,3) -- (0,3) --cycle;
\fill[lightgray] (0,0) -- (-1.5,0) -- (-1.5,1) -- (0,1) --cycle;
\fill[lightgray] (1,0) -- (3,0) -- (3,1) -- (1,1) --cycle;
\fill[gray] (0,0) -- (1,0) -- (1,1) -- (0,1) --cycle;

             \draw[step=0.5, black, dotted] (-1.5,0) grid (3,3);
\fill[black] (0.5,1.5) circle (0.3mm) (0.5,2) circle (0.3mm)
(0.5,2.5) circle (0.3mm) (0.5,3) circle (0.3mm)
(1.5,0.5) circle (0.3mm) (1.5,0) circle (0.3mm)
(2,0.5) circle (0.3mm) (3,0.5) circle (0.3mm)
(2.5,0.5) circle (0.3mm) (2.5,0) circle (0.3mm)
(-0.5,0) circle (0.3mm) (-0.5,0.5) circle (0.3mm) (-1.5,0.5) circle (0.3mm) (-1.5,0) circle (0.3mm)
(-1,0.5) circle (0.3mm);
            \foreach \x in {-1,...,3}
            {
                \foreach \y in {0,...,3}
                {
                    \fill[white,draw=black, thick] (\x,\y) circle (0.5mm);
                }
            }
  \end{tikzpicture}
\\ Case 2 &
& Case 3 \\
        \begin{tikzpicture}[scale=1]
            \fill[lightgray] (0,0) -- (1,-1) -- (1,-3/2) -- (0,-3/2) -- cycle;
            \fill[lightgray] (1,0) -- (1,1) -- (2,1) -- (2,0) -- cycle;
            \fill[lightgray] (0,1) -- (-1,1) -- (-3/2,3/2) -- (-3/2,2) -- (-1,2) -- cycle;
            \fill[gray] (0,0) -- (1,0) -- (0,1) -- cycle;
            \draw[step=0.5, black, dotted] (-1.5,-1.5) grid (2,2);
            \fill[black] (-0.5,1) circle (0.3mm)  (-1,1.5) circle
(0.3mm) (-1.5,2) circle (0.3mm) (1,0.5) circle (0.3mm)  (1.5,0.5)
circle (0.3mm) (2,0.5) circle (0.3mm)   (0.5,-0.5) circle (0.3mm)
(0.5,-1) circle (0.3mm) (0.5,-1) circle (0.3mm) (0.5,-1.5) circle
(0.3mm);
circle (0.3mm);
            \foreach \x in {-1,...,2}
            {
                \foreach \y in {-1,...,2}
                {
                    \fill[white,draw=black, thick] (\x,\y) circle (0.5mm);
                }
            }
        \end{tikzpicture}
&
\hspace{1cm}
&
 \begin{tikzpicture}[scale=1]
\fill[lightgray] (0,1) -- (1,1) -- (1,2) -- (0,2) --cycle;
\fill[lightgray] (0,0) -- (-1,0) -- (-1,1) -- (0,1) --cycle;
\fill[lightgray] (1,0) -- (2,0) -- (2,1) -- (1,1) --cycle;
\fill[lightgray] (0,0) -- (0,-1) -- (1,-1) -- (1,0) --cycle;
\fill[gray] (0,0) -- (1,0) -- (1,1) -- (0,1) --cycle;

             \draw[step=0.5, black, dotted] (-1,-1) grid (2,2);
\fill[black] 
(0.5,1.5) circle (0.3mm) (0.5,2) circle (0.3mm)
(0.5,-0.5) circle (0.3mm) (0.5,-1) circle (0.3mm)
(1.5,0.5) circle (0.3mm) (2,0.5) circle (0.3mm)
(-0.5,0.5) circle (0.3mm) (-1,0.5) circle (0.3mm);
            \foreach \x in {-1,...,2}
            {
                \foreach \y in {-1,...,2}
                {
                    \fill[white,draw=black, thick] (\x,\y) circle (0.5mm);
                }
            }
  \end{tikzpicture}
\\ Case 4 &
& Case 5 \\
\end{tabular}
\caption{Illustration of the proof of Theorem~\ref{thm:half-integral}: $Q$ is contained in the shaded area, the white dots are the elements of $\Z^2$
and the black dots are the elements of $\frac{1}{2}\Z^2$ involved in the proof.}\label{fig:cases_in_hi_proof}
\end{center}
\end{figure}
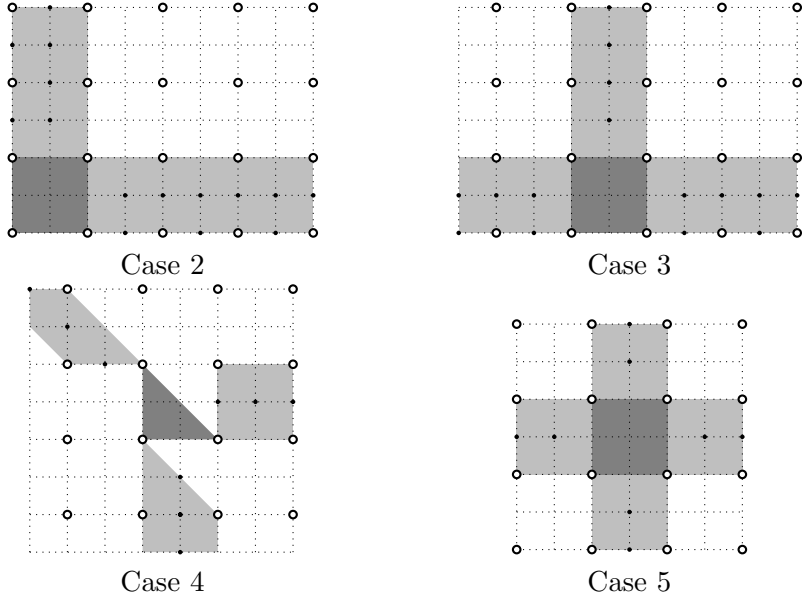

{\em Case 2: $L \in \cL_2$.}
In view of \eqref{eq:q_l_integral}, we have $[0,1]^2 \subseteq Q$. 
From the definition of $\cL_2$, we have that every point of $L \cap \left(\frac{1}{2}\Z^2 \setminus [0,1]^2\right)$
has one of the following forms:
\begin{align*}
&\left(3/2,0 \right) + k \left(1/2,0 \right), &
 \left(0,3/2\right) + k\left(0, 1/2\right), \\
& \left(3/2,1/2 \right)+ k \left(1/2,0 \right),&
 \left(1/2, 3/2\right) + k\left(0,1/2\right),
\end{align*}
where $k$ is some non-negative integer; see also Figure~\ref{fig:cases_in_hi_proof}.
Let us first assume $(\frac{3}{2}, \frac{1}{2})$ and $(\frac{1}{2}, \frac{3}{2})$ are both in $Q$. Since
$(1,1)$ is contained in the relative interior of the line segment between these two points, 
$\left[(\frac{3}{2}, \frac{1}{2}),(\frac{1}{2}, \frac{3}{2})\right]$ is contained in an edge of $Q$. Consequently, $Q \subseteq \conv \big((0,0), (2,0), (0,2)\big) = Q_2$ and
since $Q$ is $\Z^2$-maximal, this implies the equality $Q = Q_2$. 
Let us now assume that
$(\frac{1}{2}, \frac{3}{2})$ or $(\frac{3}{2}, \frac{1}{2})$ is not in $Q$. 
By applying a linear unimodular transformation that swaps $e_1$ and $e_2$ if necessary,
we can assume $(\frac{1}{2}, \frac{3}{2}) \not\in Q$. 
This implies that no point of the form
$ \left(\frac{1}{2}, \frac{3}{2}\right) + k\left(0,\frac{1}{2}\right)$ is in $Q$.
Likewise, since $(1,1) \in Q$ and $(\frac{1}{2}, \frac{3}{2}) \in [(1,1),(0,2)]$,
we have $(0,2) \not\in Q$ and consequently, $\left(0,\frac{3}{2}\right) + k\left(0, \frac{1}{2}\right) \not\in Q$ for every $k  \ge 1$. 
Hence, either $(0,3/2) \in Q$ or $Q \subseteq [0,+\infty) \times [0,1]$. 
If $(0,3/2) \in Q$, we have $Q \subseteq \conv \big((0,0), (0,3/2), (3,0)\big) \simeq Q_3$ and the $\Z^2$-maximality of $Q$ implies the equality $Q \simeq Q_3$.
Otherwise, we then have $Q \subseteq [0, +\infty) \times [0,1] \varsubsetneq \R \times [0,1]$,
a contradiction to the $\Z^2$-maximality of $Q$. 

{\em Case 3: $L \in \cL_3$.}
In view of \eqref{eq:q_l_integral}, we have $[0,1]^2 \subseteq Q$.
From the definition of $\cL_3$, we have that every point in $L \cap \left(\frac{1}{2}\Z^2 \setminus [0,1]^2\right)$ has one of the following forms:
\begin{align*}
 &\left( 3/2,0 \right) + k \left(1/2,0 \right), & & \left(-1/2,0\right) + k\left(-1/2,0\right), &\left(1/2,3/2\right) + k\left(0, 1/2\right), \\
 &\left(3/2,1/2 \right)+ k \left(1/2,0 \right), & &
 \left(-1/2, 1/2\right) + k\left(-1/2,0\right), &
\end{align*}
where $k$ is some non-negative integer; see also Figure~\ref{fig:cases_in_hi_proof}.

Analogously to Case 2, we can now restrict the values of $k$ for each of the five forms: if a point of any of the forms is in $L$ for which the value of $k$ is too large, then, because $(0,1)$ and $(1,1)$ are in the boundary of $L$, we have that $L \cap \frac{1}{2}\Z^2$ is contained in $\R \times [0,1]$ or $[0,1] \times \R$. 
In those cases, in view of \eqref{eq:q_l_halfintegral}, $Q$ is contained in a lattice-free unbounded polyhedron in $\cP(\Z^2)$. Since $Q$ is bounded,
this inclusion is strict, a contradiction
to the $\Z^2$-maximality of $Q$. This leaves a finite number of points in $\frac{1}{2}\Z^2$
which can be in $Q$ and from \eqref{eq:q_l_halfintegral}, we obtain that the vertices of $Q$ which are not in $\{0,1\}^2$ are among the points
\[ \big(-1,0\big), \big(- \frac{1}{2},0\big),  \big(-\frac{1}{2},\frac{1}{2}\big), \big(\frac{1}{2},1\big),
\big(\frac{1}{2},\frac{3}{2}\big), \big(\frac{3}{2},0\big),\big(\frac{3}{2},\frac{1}{2}\big),\big(2,0\big).
\]
It can be checked in a straightforward way that the only $\Z^2$-maximal lattice-free polygons in $\cP(\frac{1}{2}\Z^2)$ with vertices in this set
are $Q_3$ and $Q_4$.

{\em Case 4: $L \in \cL_4$.}
Here, we deduce $\Delta := \conv\big((0,0),(0,1),(1,0)\big) \subseteq Q$ from \eqref{eq:q_l_integral}. From the definition of $\cL_4$, we have that every point of $L \cap \left(\frac{1}{2}\Z^2 \setminus \Delta \right)$ has one of the following forms:
\begin{align*}
\left(1, 1/2\right) + k \left(1/2,0\right), \qquad \left(1/2,-1/2\right) + k \left(0, - 1/2\right), \qquad \left(-1/2, 1\right) + k \left(- 1/2, 1/2 \right),
\end{align*}
where $k$ is some non-negative integer; see also Figure~\ref{fig:cases_in_hi_proof}. As in the previous cases, we exploit the fact that $(0,0),(0,1),(1,0)$ lie in the boundary of $L$ to obtain
bounds on $k$ for each of the three forms: if $k$ is too large, we again deduce that $Q$ is contained in an unbounded set unimodularly equivalent to $\R \times [0,1]$. Since $Q$ is bounded, this contradicts the $\Z^2$-maximality of $Q$.
In fact, we have $k < 2$ for each of the three forms. In view of \eqref{eq:q_l_halfintegral}, this yields a finite set of points of $\frac{1}{2}\Z^2$ which can be vertices of $Q$. The straightforward verification that no lattice-free half-integral polygon with vertices in this set is $\Z^2$-maximal is left to the reader.

{\em Case 5: $L \in \cL_5$.}
In view of \eqref{eq:q_l_integral}, we have $[0,1]^2 \subseteq L$. Furthermore, it is easy to see that every point of
$(L \cap \frac{1}{2}\Z^2) \setminus \{0,1\}^2$ has one of the following forms:
\begin{align*}
\left(\frac{1}{2},\frac{1}{2}\right) \pm k\left(\frac{1}{2},0\right), \qquad \left(\frac{1}{2},\frac{1}{2}\right) \pm k\left(0, \frac{1}{2}\right),
\end{align*} 
where $k$ is an integer and $k \ge 2$; see also Figure~\ref{fig:cases_in_hi_proof}. 
One can observe that, since the points of $\{0,1\}^2$ are in the boundary of $Q$, if there exists a point of any of the four forms with $k > 2$ we have that 
$Q \subseteq L \cap \frac{1}{2}\Z^2$ is contained in $\R \times [0,1]$ or $[0,1] \times \R$. Since $Q$ is bounded, this contradicts the $\Z^2$-maximality of $Q$.
This yields $Q \subseteq Q_5$. It is straightforward to check that every $\frac{1}{2}\Z^2$-maximal integral polygon properly contained in $Q_5$ is not
$\Z^2$-maximal, which implies $Q = Q_5$.

We have shown that every $\Z^2$-maximal polyhedron in $\cP(\frac{1}{2}\Z^2)$ is unimodularly equivalent to $Q_i$ for some $i \in \{1, \ldots, 5\}$.
It is easy to see that for every $i \in \{1, \ldots, 5\}$, the polyhedron $Q_i$ is also $\R^2$-maximal as all of its edges are blocked. This yields
the equivalence of $\Z^2$-maximality and $\R^2$-maximality for lattice-free polyhedra in $\cP(\frac{1}{2}\Z^2)$.
\end{proof}

\begin{remark}\label{rem:one}
One can ask the question whether $\Z^d$-maximality and $\R^d$-maximality
are equivalent for lattice-free polyhedra in $\cP(\frac{1}{s}\Z^d)$ for given $s \in \N$. The previous theorem
shows the equivalence for $d=2$ and $s=2$.
This answers a question from \cite[Section 7.4]{wagner_diss}. Furthermore, since $\cP(\Z^2) \subseteq \cP(\frac{1}{2}\Z^2)$,
this equivalence also holds for $d=2$ and $s=1$.
\end{remark}

\section{$\Z^d$-maximal lattice-free integral polytopes with lattice width two}\label{sec:lw_two_dim_gen}%

We will now prove that every $\Z^3$-maximal lattice-free polyhedron in $\cP(\Z^3)$
is also $\R^3$-maximal. As explained in the proof strategy, we only need to prove this for polytopes, because for unbounded
polyhedra in $\cP(\Z^3)$ the equivalence of $\Z^3$-maximality and $\R^3$-maximality is straightforward to see in view of \cite[Proposition 3.1]{MR2855866}.

In the beginning of the previous section, we remarked that $\Z^2$-maximal lattice-free polytopes in $\cP(\frac{1}{2}\Z^2)$ appear naturally when considering $\Z^3$-maximal lattice-free polytopes in $\cP(\Z^3)$. 
More precisely, they appear as hyperplane sections of $\Z^3$-maximal lattice-free polytopes in $\cP(\Z^3)$, 
as we will explain in detail in this section. However, where possible, we do not restrict ourselves to dimension three 
but instead present the argumentation concerning lattice-free integral polytopes with lattice width two in arbitrary dimension $d \ge 3$. 
Our general setting is the following: let $P \in \cP(\Z^d)$ be a $\Z^d$-maximal lattice-free polytope with lattice width two. 
Then there exists a primitive vector $u \in \Z^d \setminus \{o\}$ such that $\width{P,u} = 2$. Changing the coordinates using an appropriate linear unimodular
transformation, we can always assume $u$ to be $e_d$ and using some translation by an integral vector if necessary, we can furthermore without loss
of generality assume $P \subseteq \R^{d-1} \times [-1,1]$.

\subsection*{Properties of faces and cross-sections}

If $P \in \cP(\Z^d)$ is contained in $\R^{d-1} \times [-1,1]$, then
the integral points of $P$ are contained in $\R^{d-1} \times \{-1,0,1\}$ and in particular, the latter set contains all vertices of $P$.
Note that if moreover $P$ has lattice width two, then $\R^{d-1} \times \{0\}$ meets the interior of $P$ and since $P$ is lattice-free,
this allows us to derive a characterization of $P \cap \R^{d-1} \times \{0\}$. 
We then provide relations between $P \cap \R^{d-1} \times \{0\}$ and $F(P,e_d), F(P,-e_d)$.
For the characterization of $P \cap \R^{d-1} \times \{0\}$, we need the following lemma.

	\begin{lemma}\label{prop:aux_basic_convexity}
		Let $H$ be an affine subspace of $\R^d$. Let $A, B$ be convex subsets of $\R^d$. Then
                the following statements hold:
\begin{itemize}
 \item[(a)]\label{A:subset:B:arbitrary} If $A \subseteq H$, then $
			H \cap \conv(A \cup B) = \conv\big(A \cup (H \cap B)\big).$
 \item[(b)]\label{cross-section:of:prismatoid} If $A \not\subseteq H$, $B \not\subseteq H$ and $(1-\lambda) A + \lambda B \subseteq H$ for some $0 < \lambda < 1$, then
\[			H \cap \conv( A \cup B) = (1-\lambda) A + \lambda B.
		\]
\end{itemize}
	\end{lemma}

	\begin{proof}
		{\em (a):} Let $A \subseteq H$. It suffices to verify the inclusion $H \cap \conv(A \cup B) \subseteq \conv\big(A \cup (H \cap B)\big)$; 
		the reverse inclusion is trivial, as both $A$ and $H \cap B$ are subsets of $H \cap \conv(A \cup B)$. 
		Consider an arbitrary $x \in H \cap \conv(A \cup B)$. 
		That is, $x$ is a point of $H$ that can be written as a convex combination of some $a \in A$ and $b \in B$. If $x = a$, then in particular $x \in A$ and we are done. If not, then $b$ lies on the line passing through $x$ and $a$. Since $x,a \in H$, this line lies in $H$ and therefore, we also have $b \in H$. 

		{\em (b):} Let $A \not\subseteq H$ and $B \not\subseteq H$. 
It suffices to verify the inclusion $H \cap \conv(A \cup B) \subseteq (1-\lambda) A + \lambda B$, as the reverse inclusion is trivial. 
Obviously, one has $\conv(A \cup B) = \bigcup_{\mu \in [0,1]} ((1-\mu) A + \mu B))$. 
Now we fix an affine function $f : \R^d \rightarrow \R$ such that $H = \setcond{x \in \R^d}{f(x) = 0}$. 
Since $H \neq \R^d$, the function $f$ is not identically equal to zero. 
Applying $f$ to $(1-\lambda) A + \lambda B \subseteq H$, we arrive at $(1-\lambda) f(A) + \lambda f(B) \subseteq f(H)$, 
where by the choice of $f$, one has $f(H) = \{0\}$. In view of $0 < \lambda < 1$, the latter shows that $f(A)$ and $f(B)$ must be singletons. 
We interpret the singletons $f(A), f(B), f(H)$ as real values, where $f(H) = 0$. Let $f(A) = \alpha$ and $f(B) = \beta$. 
Since $A$ and $B$ are not subsets of $H$, we have $\alpha,\beta \neq 0$ and in view of $(1- \lambda) \alpha + \lambda \beta = 0$, 
we have $\alpha \neq \beta$.
Hence, for $\mu \in [0,1]$, one has $(1 - \mu) \alpha + \mu \beta = 0$ if and only if $\mu = \lambda$ and, 
thus, we have shown that $(1 - \mu) A + \mu B$ is disjoint with $H$ unless $\mu = \lambda$. Consequently, we obtain (b).
\end{proof}

\begin{theorem}\label{lem:mega_lemma}
 Let $P \in \cP(\Z^d)$ be a $\Z^d$-maximal lattice-free polytope such that $P \subseteq \R^{d-1} \times [-1,1]$ and $\width{P,e_d} = 2$. 
For $i \in \{-1,0,1\}$, let $P_i := \setcond{x \in \R^{d-1}}{(x,i) \in P}$. Then the following statements hold:
\begin{enumerate}[label=(\alph*)]
\item\label{item:part_one}	
$P_0$ belongs to $\cP(\frac{1}{2}\Z^{d-1})$ and is $\Z^{d-1}$-maximal lattice-free.
\item\label{item:part_two} 
$P_{-1}$ and $P_1$ belong to $\cP(\Z^{d-1})$ and have the following properties relative to $P_0$:
\begin{enumerate}[label=(b\arabic*)]
			\item\label{eq:layers_inclusion_up} $P_1$ and $P_{-1}$ are integral polytopes satisfying
\begin{align}
 \label{eq:mink_sum_up} 
P_1 + P_{-1} \subseteq 2P_0. 
\end{align}
			\item\label{eq:layers_maximality} The pair $(P_1, P_{-1})$ satisfies the following maximality condition:
for polytopes $R_{-1}, R_1 \in \cP(\Z^{d-1})$, the inclusions $P_1 \subseteq R_1$,
$P_{-1} \subseteq R_{-1}$ and $R_1 + R_{-1} \subseteq 2P_0$ imply $R_1 = P_1$ and $R_{-1} = P_{-1}$.
			\item\label{eq:layers_inclusion_low} $P_1$ and $P_{-1}$ satisfy
\begin{align}
 \label{eq:mink_sum_lhs} 
 2\conv(\vertset{P_0} \setminus \Z^{d-1}) \subseteq P_1 + P_{-1}. 
\end{align}
\end{enumerate}
\end{enumerate}
\end{theorem}

	\begin{proof}
Throughout the proof, we will use the following notation: $H$ denotes the hyperplane $\R^{d-1} \times \{0\}$ and for $i \in \{-1,0,1\}$, we write $\overline{P}_i := P_i \times \{i\}$.
We will repeatedly make use of the fact that $P \subseteq \R^{d-1} \times [-1,1]$ and hence, 
$P \cap \Z^d \subseteq \Z^{d-1} \times \{-1,0,1\}.$
In particular, taking into account $P \in \cP(\Z^d)$, we have
\begin{align}
 \vertset{P} \subseteq \oq{-1} \cup \oq{0} \cup \oq{1}. \label{eq:vertices_in_layers}
\end{align}
The assumption $\width{P,e_d} = 2$ implies $P_1 \neq \emptyset$ and $P_{-1} \neq \emptyset$.

		Consider
		\begin{align*}
		A := \conv(H \cap \vertset{P}) \qquad \textand \qquad  B := \conv\big(\oq{-1} \cup \oq{1}\big).\end{align*}
		From \eqref{eq:vertices_in_layers}, it is clear that $\conv(A \cup B) =  P$.
		Hence, in view of Lemma~\ref{prop:aux_basic_convexity}(a), we get \begin{align}\label{eq:oline_q0} \oq{0} = H \cap \conv(A \cup B)= \conv\big(A \cup (H \cap B)\big), \end{align} 
		where the second equality is a consequence of Lemma~\ref{prop:aux_basic_convexity}(a).

{\em Assertion (a):}
		We first prove that $P_0$ belongs to $\cP(\frac{1}{2}\Z^d)$, which is equivalent to $\vertset{\oq{0}} \subseteq \frac{1}{2}\Z^{d-1} \times \{0\}$. 
		In \eqref{eq:oline_q0}, the set $A$ is an integral polytope, possibly empty. 
		Thus, it suffices to verify that the vertices of $H \cap B$ belong to $\frac{1}{2} \Z^{d-1} \times \{0\}$. In fact, Lemma~\ref{prop:aux_basic_convexity}(b) yields 
		\begin{equation}
			\label{B:cap:H:eq}
			H \cap B = H \cap \conv\big(\oq{-1} \cup \oq{1}\big) = \frac{1}{2} \oq{-1} + \frac{1}{2} \oq{1}.
		\end{equation}
  Thus, since $\oq{-1}$ and $\oq{1}$ are integral polytopes, $H \cap B$ is half-integral and so is $P_0$. 
		
		Next, we show that $P_0$ is lattice-free. By \cite[Corollary 6.5.1]{MR1451876} we have that $\relintr{P \cap H} = \intr{P} \cap H$ and thus
    \begin{align*}
        \relintr{\oq{0}} \cap \Z^d &= \relintr{P \cap H} \cap \Z^d = \intr{P} \cap H \cap \Z^d\\ &\subseteq \intr{P} \cap \Z^d = \emptyset.
    \end{align*}
    Hence, $\relintr{ \oq{0} }$ does not contain points of $\Z^{d-1} \times \{0\}$ and thus, $ P_0 $ is lattice-free.

		To complete the proof of \ref{item:part_one}, it remains to show that $P_0$ is $\Z^{d-1}$-maximal. 
    Let $ p \in \Z^{d-1} \times \{0\} \setminus \oq{0} $.
    As~$ p \notin P $ and~$ P $ is $ \Z^d $-maximal lattice-free, there exists some $ z  \in \Z^d$ in $\intrbig{\conv(P \cup \{p\})}$.
    Since~$ P $ is lattice-free, we have~$ z = \lambda p + (1-\lambda) x $ for some~$ 0 < \lambda < 1 $ and~$ x \in P $.
    Let $x_d, p_d, z_d$ denote the $d$-th component of $x,p,z$, respectively.
    Because~$ x_d \in [-1,1] $ and~$ p_d = 0 $, we obtain~$ -1 < z_d < 1 $ and hence~$ z_d = 0 $.
    This implies
    \begin{align*}
        z \in \relintrbig{\conv(P \cup \{p\})} \cap H & = \relintrbig{\conv(P \cup \{p\}) \cap H} \\
        & = \relintrbig{\conv(\oq{0} \cup \{p\})},
    \end{align*}
    where the first equality follows from \cite[Corollary 6.5.1]{MR1451876} and the second equality follows from Lemma~\ref{prop:aux_basic_convexity}(a).
    Hence,~$ P_0 $ is $ \Z^{d-1} $-maximal lattice-free.

{\em Assertion (b)}:
		Obviously, $P_{-1}, P_1$ are integral polytopes because $\oq{-1}, \oq{1}$ are faces of the integral polytope $P$.

\emph{(b1):}
		Clearly, $\frac{1}{2}\oq{-1} + \frac{1}{2}\oq{1}$ is a subset of both $P$ and $H$. This yields $\frac{1}{2}\oq{-1} + \frac{1}{2}\oq{1} \subseteq \oq{0}$
		and hence assertion (b1) follows.

		\emph{(b2):} Let $R_1, R_{-1} \subseteq \R^{d-1}$ be integral polytopes such that $R_1 + R_{-1} \subseteq 2P_0$ and $P_1 \subseteq R_1$, $P_{-1} \subseteq R_{-1}$. 
We write $\overline{R}_i = R_i \times \{i\}$ for $i  \in \{-1,1\}$. Consider the integral polytope
\[R := \conv\big(\overline{R}_{-1} \cup \overline{R}_{1} \cup \oq{0}\big),\] 
which is contained in $\R^{d-1} \times [-1,1]$ and contains $P$.
In view of Lemma~\ref{prop:aux_basic_convexity}(a), we have
\begin{align*}
R \cap H = \conv\Big(\oq{0} \cup \conv \big(\big(\overline{R}_1 \cup \overline{R}_{-1}\big) \cap H\big) \Big).
\end{align*}
Applying Lemma~\ref{prop:aux_basic_convexity}(b) yields
\[\conv \big(\overline{R}_1 \cup \overline{R}_{-1}\big) \cap H = \big(\frac{1}{2}R_1 + \frac{1}{2}R_{-1}\big) \times \{0\} \subseteq \oq{0},\] where the inclusion follows from $R_1 + R_{-1} \subseteq 2P_0$. Hence, $R \cap H = \oq{0}$.
		Therefore, one has $\relintr{H \cap R} \cap \Z^d = \relintr{\oq{0}} \cap \Z^d$.
		By \ref{item:part_one}, $P_0$ is lattice-free and hence, we have $\relintr{\oq{0}} \cap \Z^d = \emptyset$. 
		Since $R \cap \Z^d \subseteq \Z^{d-1} \times \{-1,0,1\}$, where $R \cap \Z^{d-1} \times \{-1,1\}$ is contained in the boundary of $R$,
		this shows that $R$ is a lattice-free integral polytope containing $P$. Since $P$ is $\Z^d$-maximal, we have $P = R$ and in particular,
		$R_1 = P_1$ and $R_{-1} = P_{-1}$.

		\emph{Assertion (b3):} Let $v \in \R^{d-1}$ be a non-integral vertex of $P_0$ and let $\overline{v} := (v,0) \in \R^d$. Then $\overline{v}$
		is not integral and since
		$P$ is an integral polytope, $\overline{v}\not\in \vertset{P}$. On the other hand, since $v$ is a vertex of $P_0$, $\overline{v}$ cannot be written as convex combination of any points of $P \cap H$. Thus, we have $v \not\in A$ and consequently, in view of \eqref{eq:oline_q0}, we also have $v \in H \cap B$. By \eqref{B:cap:H:eq}, there exist $x_1 \in \oq{1}$, $x_{-1} \in \oq{-1}$ such that $\overline{v} = \frac{1}{2}x_1 + \frac{1}{2}x_{-1}$. 
Accordingly, $2v \in P_1 + P_{-1}$. This proves the assertion.
	\end{proof}

\subsection*{Analysis of special cross-sections}

From Figure~\ref{fig:s2d2} on page ~\pageref{fig:s2d2}, one can make the following observation: three of the four $\Z^2$-maximal
half-integral polytopes are simplices with the property that all of their integral vertices are unimodular\index{vertex!unimodular}. It turns out that if $P_0$ in the setting of Theorem~\ref{lem:mega_lemma} is a simplex with this property, one can strengthen the assertions about the relation between $P_1,P_{-1}$ and $P_0$
presented in Theorem~\ref{lem:mega_lemma}. 

\begin{lemma}\label{lem:hom_minimal}
Let $Q \in \cP(\Z^d)$ be a polytope such that $Q$ is not a singleton and let $v$ be a vertex of $Q$. Let $g := \gcd\big(\setcond{\sprod{w}{e_i}}{w \in \vertset{Q-v}, \, i \in \{1, \ldots, d\}}\big)$, i.e., $g$ is the greatest common divisor of all the components of the vertices of $Q-v$.
Let $\tq := \frac{1}{g} (Q - v) \in \cP(\Z^d)$. Then every non-negative homothetic copy $Q'$ of $Q$ belonging to $\cP(\Z^d)$
has the form $Q' = a\tq + b$, where $a \in \N \cup \{0\}$ and $b \in \Z^d$.
\end{lemma}

\begin{proof}
Let $Q' \in \cP(\Z^d)$ be a non-negative homothetic copy of $Q$.
Since obviously $Q'$ is then also a non-negative homothetic copy of $\tq$, there exist
$a \ge 0$ and $b \in \R^d$ such that $Q' = a \tq + b$. 
By definition of $\tq$, one of its vertices is $o$.
Hence, $b \in \vertset{Q'}$ and therefore, in view of $Q' \in \cP(\Z^d)$, we have $b \in \Z^d$.
Thus, $Q' - b \in \cP(\Z^d)$ and hence from
$\vertset{Q' -b} = \frac{a}{g}(\vertset{Q} - v) \subseteq \Z^d$ we get
$\frac{a}{g}\sprod{w'-v}{e_i} \in \Z$ for every $w' \in \vertset{Q}$ and every $i \in \{1, \ldots, d\}$. 
Suppose now that $a$ is not an integer,
i.e. $a$ has a representation as $p/q$ with $p \in \N$ and $q \in \N \setminus \{1\}$ being relatively prime. 
Then $q$ divides $\frac{\sprod{w' - v}{e_i}}{g}$ for every $w' \in \vertset{Q}$ and every $i \in \{1, \ldots, d\}$ and hence,
$q g$ divides $\sprod{w}{e_i}$ for every $w \in \vertset{Q-v}$ and every $i \in \{1, \ldots, d\}$, which contradicts the definition of $g$.
\end{proof}

\begin{theorem}\label{lem:homothetic}
 Let $P, P_0, P_1, P_{-1}$ be as in Theorem~\ref{lem:mega_lemma}. If $P_0$ is a $(d-1)$-dimensional simplex such that all integral vertices of $P_0$ are unimodular,
then the following statements hold:
\begin{itemize}
 \item[(a)] $P_1$ and $P_{-1}$ are non-negative homothetic copies of $P_0$,
 \item[(b)] $P_1 + P_{-1} = 2P_0$.
\end{itemize}
\end{theorem}

\begin{proof}
{\em Assertion (a)}. Let $u_1, \ldots, u_d \in \Z^{d-1} \setminus \{o\}$ be primitive vectors such that $U(P_0) = \{u_1, \ldots, u_d\}$. We circumscribe a non-negative homothetic copy of $P_0$
around $P_1$. More precisely, we consider the polytope
\[P'_1 := \setcond{x \in \R^{d-1}}{\sprod{x}{u_j} \le h(P_1,u_j) \text{ for every } j \in \{1, \ldots, d\}},\]
which contains $P_1$ and is a non-negative homothetic copy of $P_0$ since $P'_1$ is a simplex
having the same set of outer normal facet vectors as $P_0$. 
We will now prove that $P'_1 = P_1$, which implies that $P_1$ is a non-negative homothetic copy of $P_0$.

Let $v \in \vertset{P'_1}$.
Appropriately reindexing $u_1, \ldots, u_d$, we assume 
\[
\sprod{v}{u_i} = h(P'_1,u_i) \text{ for every } i \in \{1, \ldots, d-1\}.
\]
Denote by $v_0$ the vertex of $P_0$ with the property 
\[
\sprod{v_0}{u_i} = h(P_0,u_i) \text{ for every } i \in \{1, \ldots, d-1\}. 
\]
Let us first consider the case $v_0 \not\in \Z^{d-1}$. 
In this case, \eqref{eq:mink_sum_lhs} yields $2v_0 \in P_1 + P_{-1}$. 
Thus, taking into account that $2v_0$ is a vertex of $2P_0$ and $P_1 + P_{-1} \subseteq 2P_0$, we get that $2v_0$ is a vertex
of $P_1 + P_{-1}$. 
Note that $P_1 + P_{-1} = \conv(\vertset{P_1}) + \conv(\vertset{P_{-1}}) = \conv(\vertset{P_1} + \vertset{P_{-1}})$ (see \cite[Theorem 1.1.2]{MR1216521}) and hence,
$\vertset{P_1 + P_{-1}} \subseteq \vertset{P_1} + \vertset{P_{-1}}$. Therefore,
there exist $v_1 \in \vertset{P_1}$, $v_{-1} \in \vertset{P_{-1}}$ with $2v_0 = v_1 + v_{-1}$.
In view of $P_1 + P_{-1} \subseteq 2P_0$, we have $\cone(u_1, \ldots, u_{d-1}) = N(2P_0,2v_0) \subseteq N(P_1 + P_{-1}, 2v_0)$. Furthermore, 
we have $N(P_1 + P_{-1}, 2v_0) = N(P_1,v_1) \cap N(P_{-1},v_{-1})$; see \cite[Theorem 2.2.1]{MR1216521}. Together, this yields
$\cone(u_1, \ldots, u_{d-1}) \subseteq N(P_1,v_1)$.
In other words, $u_1, \ldots, u_{d-1}$
are outer normal vectors of $P_1$ at $v_1$. Hence, we have
\[\sprod{v_1}{u_i} = h(P_1,u_i) \text{ for every } i \in \{1, \ldots, d-1\},\] 
and thus, $v$ and $v_1$ are defined by the same (uniquely solvable) system of linear equalities. Hence, $v = v_1$ and in particular, $v$ is integral.

We now treat the case that $v_0 \in \Z^{d-1}$, i.e., $v_0$ is a vertex of $P_0$. 
In this case, since all integral vertices of $P_0$ are unimodular, $u_1, \ldots, u_{d-1}$ form a basis of $\Z^{d-1}$,
i.e., the matrix with rows $u_1, \ldots, u_{d-1}$ has determinant one.
Since the vertex $v$ of $P'_1$ is the unique solution of the system of linear inequalities given by
\[\sprod{u_i}{v} = h(P_1,u_i) \text{ for every } i \in \{1, \ldots, d-1\},\]
and $h(P_1,u_i)$ is integral for every $i \in \{1, \ldots, d-1\}$, it thus follows
that $v$ is integral.

We have now shown that all vertices of $P'_1$ are integral and hence, $P'_1$ is an integral polytope containing $P_1$. Observe now that
$P'_1 + P_{-1} \subseteq 2P_0$, since for $x \in P'_1$ and $y \in P_{-1}$ and every $j \in \{1, \ldots, d\}$, one has 
\[
 \sprod{x+y}{u_j} = \sprod{x}{u_j} + \sprod{y}{u_j} \le h(P_1,u_j) + h(P_{-1},u_j) = h(P_1 + P_{-1}, u_j) \le h(2P_0,u_j).
\]
Here, the last inequality follows from $P_1 + P_{-1} \subseteq 2P_0$.
In view of Theorem~\ref{lem:mega_lemma}\ref{eq:layers_maximality} with $R_1 = P'_1$ and $R_{-1} = P_{-1}$, we then immediately obtain $P_1 = P'_1$. This proves that $P_1$ is a non-negative homothetic copy of $P_0$. To show that $P_{-1}$ is a non-negative homothetic copy of $P_0$, one can argue in exactly the same way by interchanging the roles of $P_1$ and $P_{-1}$.

{\em Assertion (b).} 
Let $v$ be any vertex of $2P_0$. We define
$\tp := \frac{1}{g}(2P_0 - v)$ using $g$ as in Lemma~\ref{lem:hom_minimal} with $Q = 2P_0$.  By (a),
$P_1$ and $P_{-1}$ are homothetic copies of $P_0$. Thus, Lemma~\ref{lem:hom_minimal} yields
$P_1 = a \tp + b$ and $P_{-1} = a' \tp + b'$, where $a,a' \in \N \cup \{0\}$
and $b,b' \in \Z^{d-1}$.
By Theorem~\ref{lem:mega_lemma}\ref{eq:layers_inclusion_up}, we have $P_1 + P_{-1} \subseteq 2P_0$.
Summarizing these arguments, we get
\[
a \tp + a' \tp + (b + b') \subseteq 2P_0 = g \tp + v.
\]
Since $g, a, a' \ge 0$ and since $\tp$ is $(d-1)$-dimensional, we have $a + a' \le g$, i.e. $0 \le a' \le g - a$.
Then since $g,a \in \N \cup \{0\}$, the polytope $R_{-1} := (g-a) \tp + b'$ is in $\cP(\Z^{d-1})$ and contains $P_{-1}$.
In view of $P_1 + R_{-1} \subseteq 2P_0$ and Theorem~\ref{lem:mega_lemma}\ref{eq:layers_maximality}, we obtain $R_{-1} = P_{-1}$ and by this, $a' = g - a$
and $a + a' = g$.
This yields $g \tp + (b + b') = g \tp + v$ and implies $b + b' = v$.
\end{proof}

\subsection*{Enumeration for lattice width two and dimension three}\label{sec:lw_two_dim_three}

In view of the results of this section,
all possible choices for $P_0$ are contained in the list of polygons given
in Theorem~\ref{thm:half-integral}. Based on this, we 
can now construct all $\Z^3$-maximal lattice-free polytopes $P$ in $\cP(\Z^3)$ with lattice width two.
We determine all such $P$ and see that all of them are also $\R^3$-maximal.

\begin{proposition}\label{prop:lattice_width_two}
 Let $P \in \cP(\Z^3)$ be a lattice-free polytope with $\lw{P} = 2$. Then, $P$ is $\Z^3$-maximal if and only if $P$ is
unimodularly equivalent to one of the following polytopes (see Figure~\ref{fig:s1d3_width_two}):
\begin{itemize}
\item[] the simplex $\maximizer_{4,6} := \conv(-2e_1, 4e_1, e_1 + 3e_2, 2e_3)$,
\item[] the simplex $\maximizer_{4,4} := \conv(o, 4e_1, 4e_2, 2e_3)$,
\item[] the simplex $\maximizer_{4,2} := \conv(-e_1+e_2, e_1 + 3e_2, 2e_3, 2e_1 - 2e_2 + 2e_3)$,
\item[] the simplex $\maximizer'_{4,4} := \conv(-e_1,3e_1,e_1 + 4e_2, 2e_3)$,
\item[] the pyramid $\maximizer_{5,4} := \conv(e_1 - e_2, -e_1 + e_2, 3e_1 + e_2, e_1 + 3e_2, 2e_3)$,
\item[] the prism $\maximizer_{5,2} := \conv(-e_1,e_1,2e_2,2e_3,2e_1 + 2e_3,e_1 + 2e_2 + 2e_3)$,
\item[] the parallelepiped $\maximizer_{6,2} := \conv(o, -e_1+e_2, 2e_2, e_1+e_2, 2e_3, e_1+e_2+2e_3, 2e_1 + 2e_3, e_1 -e_2 + 2e_3)$.
\end{itemize}
In particular, $P$ is
 $\Z^3$-maximal if and only if $P$ is $\R^3$-maximal.
\end{proposition}

\begin{proof}
We remark that throughout the proof, whenever we claim that two polytopes which are explicitly given are unimodularly equivalent, we omit
the verification of this claim. This verification can be carried out either by hand, using elementary linear algebra, or using specialized computer software.

Let $P$ be $\Z^3$-maximal.
By applying a unimodular transformation, we can assume $P \subseteq \R^2 \times [-1,1]$ and $\width{P,e_3} = 2$. 
For $i \in \{-1,0,1\}$, let $P_i := \setcond{x \in \R^2}{(x,i) \in P}$. 
In view of Theorem~\ref{thm:half-integral}  and Theorem~\ref{lem:mega_lemma}\ref{item:part_one}, we have that $P_0$ is unimodularly equivalent to one of the polygons
$Q_2, Q_3, Q_4, Q_5$ as in Theorem~\ref{thm:half-integral}. By applying another unimodular transformation to $P$, we assume that 
$P_0 = Q_i$ for some $i \in \{2,3,4,5\}$.
Furthermore, $P_1$ and $P_{-1}$ satisfy the conditions of Theorem~\ref{lem:mega_lemma}\ref{item:part_two}.

{\em Case 1: $P_0$ is a simplex.}
That is, $P$ is $Q_2,Q_3$ or $Q_4$. These latter three triangles have the property
that each of their integral vertices is unimodular. 
In view of Theorem~\ref{lem:homothetic}, we obtain that if $P_0 \in \{Q_2,Q_3,Q_4\}$, then $P_1,P_{-1} \in \cP(\Z^2)$
are non-negative homothetic copies of $P_0$ such that $P_1 + P_{-1} = 2P_0$. In order to apply Lemma~\ref{lem:hom_minimal},
we fix a vertex $v \in \cP(\Z^2)$ of $2P_0$ and consider $\tp := \frac{1}{g}(2P_0 - v)$, where
$g := \gcd(\{w_1,w_2,w_3,w_4\})$ with $(w_1,w_2)$ and $(w_3,w_4)$ being the elements of $\vertset{2P_0 -v} \setminus \{o\}$.
Then there exist $a, a' \in \N \cup \{0\}$ and $b,b' \in \Z^2$ such that $P_1 = a\tp + b$, $P_{-1} = a' \tp + b'$ and
$2P_0 = g\tp + v = (a + a') \tp + (b+b')$.
As shown in the proof of Theorem~\ref{lem:homothetic}(b), we have $b + b' = v$, i.e. $b' = v - b$, and $a + a' = g$. It is also obvious that reversing the roles of $P_1$ and $P_{-1}$
corresponds to reflecting $P$ with respect to $\R^2 \times \{0\}$, which is a unimodular transformation. Without loss of generality, we can
therefore assume $a \le a'$. Moreover, observe that we can also assume $b = o$ (and hence, $b' = v$). To see this, let $b = (b_1,b_2)$ and
let $\varphi \colon \R^3 \to \R^3$ be the linear mapping given by $\varphi(x_1,x_2,x_3) = (x_1 - b_1 x_3, x_2 - b_2 x_3, x_3)$.
Then $\varphi$ is indeed a unimodular transformation mapping $(a\tp + b) \times \{1\}$ onto $a\tp \times \{1\}$ and $(a'\tp + v - b) \times \{-1\}$
onto $(a'\tp + v) \times \{-1\}$ while leaving $g \tp \times \{0\}$ unchanged.

We now proceed as follows: for given $P_0$, we select a vertex $v$ and determine $g$ and $\tp$. Then
for every pair of non-negative integers $a,a'$ with $a \le a'$ and $a+a' = g$, we check the polytope
$P = \conv\big((a'\tp + v) \times \{-1\} \cup a\tp \times \{1\}\big)$ for $\Z^3$-maximality. If $P$ is $\Z^3$-maximal, we show that it is unimodularly
equivalent to one of the seven polytopes given in the formulation of the theorem.

For the three cases, we select the vertex $v$ and obtain $\tp$ as follows:
\begin{align*}
&\text{Case 1.1:} & P_0 = Q_2, & \qquad v = (0,0), & & \tp = \conv\big((0,0), (1,0), (0,1)\big); \\
&\text{Case 1.2:} & P_0 = Q_3, & \qquad v = (-2,0), & &\tp = \conv\big((0,0),(2,0),(1,1)\big);\\
&\text{Case 1.3:} & P_0 = Q_4, & \qquad v = (-1,0), & & \tp = \conv\big((0,0),(2,0),(1,2)\big).
\end{align*}

\begin{figure}[h!]
\begin{center}
\begin{tabular}{ccc}
\begin{tikzpicture}
\draw[fill = lightgray] (0,0) -- (0,1) -- (1,0) --cycle;
\draw[black] (0,0) -- (0,2) -- (2,0) -- cycle;
 		\draw[fill=black] (0,0) circle (0.05);
 		\draw[fill=black] (0,1) circle (0.05);
 		\draw[fill=black] (1,0) circle (0.05);
 		\draw[fill=black] (1,1) circle (0.05);
		\draw[fill=black] (2,0) circle (0.05);
 		\draw[fill=black] (0,2) circle (0.05);
 		\node[below] at (0,0) {\scriptsize{$(0,0)$}};
\end{tikzpicture}
&
\begin{tikzpicture}
\draw[fill = lightgray] (0,0) -- (2,0) -- (1,1) --cycle;
\draw[black] (-1,0) -- (2,0) -- (1/2,3/2) -- cycle;
 		\draw[fill=black] (0,0) circle (0.05);
 		\draw[fill=black] (0,1) circle (0.05);
 		\draw[fill=black] (1,0) circle (0.05);
 		\draw[fill=black] (1,1) circle (0.05);
		\draw[fill=black] (2,0) circle (0.05);
 		\draw[fill=black] (-1,0) circle (0.05);
 		\node[below] at (0,0) {\scriptsize{$(0,0)$}};
\end{tikzpicture}
&
\begin{tikzpicture}
\draw[fill = lightgray] (0,0) -- (2,0) -- (1,2) --cycle;
\draw[black] (-1/2,0) -- (3/2,0) -- (1/2,2) -- cycle;
 		\draw[fill=black] (0,0) circle (0.05);
 		\draw[fill=black] (0,1) circle (0.05);
 		\draw[fill=black] (1,0) circle (0.05);
 		\draw[fill=black] (1,1) circle (0.05);
		\draw[fill=black] (2,0) circle (0.05);
 		\draw[fill=black] (1,2) circle (0.05);
 		\node[below] at (0,0) {\scriptsize{$(0,0)$}};
\end{tikzpicture}
\\
Case 1.1 & Case 1.2 & Case 1.3
\end{tabular}
\caption{Illustration of $P_0$ and $\tp$ (shaded) in the Cases 1.1-1.3 of the proof of Proposition~\ref{prop:lattice_width_two}.}
\end{center}
\end{figure}
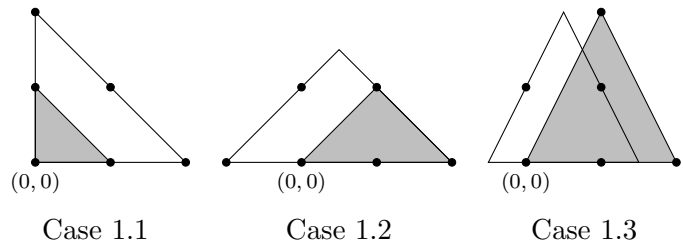

{\em Case 1.1: $P_0 = Q_2$.} Then $2P_0 = 4\tp$. 
In view of $a \le a'$, we have three cases to distinguish:
first, $a = 0$ and $a' = 4$, second, $a = 1$ and $a' = 3$, and third, $a = a' = 2$.
In the first case, $P$ is the simplex $\maximizer_{4,4} - e_3$. 
In the second case, $P$ is $\conv\big((3\tp - e_3) \cup (\tp + e_3)\big)$, which is properly contained in the integral lattice-free polytope $\conv(o, 3e_1, 3e_2, 3e_3) - e_3$ and is thus
not $\Z^3$-maximal. If $a = 2$, then $P = \conv(o, 2e_1, 2e_2) \times [-1,1]$, which is properly contained in the integral lattice-free polyhedron $\conv(o, 2e_1, 2e_2) \times \R$
and thus not $\Z^3$-maximal.

{\em Case 1.2: $P_0 = Q_3$.} Then $2P_0 = 3\tp - (2,0)$.
There are two cases to distinguish: $a = 0$, $a' = 3$ and $a = 1$, $a' = 2$. In the former case,
$P$ is the simplex $\maximizer_{4,6} - e_3$.  
In the latter case, \[P = \conv\big((-2,0,-1),(0,0,-1),(-1,1,-1),(0,0,1),(2,0,1),(1,1,1)\big),\]
which is properly contained in $\conv\big((-2,0,-1),(0,0,-1),(-1,1,-1),(0,0,2)\big)$. This implies that $P$ is not $\Z^3$-maximal.

{\em Case 1.3: $P_0 = Q_4$.} Then $2P_0 = 2\tp - (1,0)$. 
There are two cases to distinguish: $a = 0$, $a' = 2$ and $a = a' = 1$.
In the former case, $P$ is the simplex $\maximizer'_{4,4} - e_3$.
If $a = a' = 1$, then
$P = \conv\big((\tp - (1,0)) \times \{-1\} \cup \tp \times \{1\}\big) = \maximizer_{5,2} - e_3$.

{\em Case 2: $P_0 = Q_5$.}
Then we have $2P_0 = \conv(\pm (2,0), \pm (0,2)) + (1,1)$. 
Since $P_0$ does not have integral vertices, combining
\eqref{eq:mink_sum_up} and \eqref{eq:mink_sum_lhs} we obtain $P_1 + P_{-1} = 2P_0$. 
Consequently, for every $u \in \Z^2 \setminus \{o\}$, one has
$F(2P_0, u) = F(P_1 + P_{-1}, u) = F(P_1,u) + F(P_{-1},u)$; see \cite[Theorem 1.7.5(c)]{MR1216521}. 
The latter equality shows that every one-dimensional face of $P_1$ and $P_{-1}$ is parallel to an edge of $P_0$. In other words, if $P_1$ or $P_{-1}$ is not a singleton, each of its facets is parallel to one of the line segments $S_1 := [(0,0),(1,1)]$ and $S_2 := [(0,0),(1,-1)]$
and thus, there exist non-negative integers $k_1,k_2,l_1,l_2$ and vectors $v,w \in \Z^2$ such that 
\[P_1 = k_1 S_1 + k_2 S_2 + v \text{ and } P_{-1} = l_1 S_1 + l_2 S_2 + w.\]
Observe also that one has $2P_0 = 2S_1 + 2S_2 + (-1,1)$, which in view of $P_1 + P_{-1} = 2P_0$ implies
$v + w = (-1,1)$ and $k_1 + l_1 = k_2 + l_2 = 2$. As argued in Case 1, reversing the roles of $P_1$ and $P_{-1}$ results in reflecting $P$
with respect to $\R^2 \times \{0\}$ and hence, we can assume $k_1 \le l_1$, which in view of $k_1 + l_1 = 2$ implies $k_1 \le 1$. 
Likewise, in Case 1 we have fixed $b = (0,0)$ and here, we can do the same for $v$, which implies $w = (-1,1)$. Since $l_1$ and $l_2$ are determined by the choice of $k_1$ and $k_2$, respectively,
we have the following six cases depending on the choice of $(k_1,k_2) \in \{0,1\} \times \{0,1,2\}$.
\begin{compactitem}
 \item
If $(k_1,k_2) = (0,0)$, then $P = \conv(2P_0 \times \{-1\} \cup \{e_3\}) = \maximizer_{5,4} - e_3$. 
 \item
If $(k_1,k_2) = (0,1)$, then $P = \conv\big((2S_1 + S_2 + (-1,1)\big) \times \{-1\} \cup S_2 \times \{1\})$. Here, $P$ is a prism unimodularly equivalent to $\maximizer_{5,2}$.
 \item
If $(k_1,k_2) = (0,2)$, then $P = \conv\big((2S_1 +(-1,1)) \times \{-1\} \cup (2S_2) \times \{1\}\big) = \maximizer_{4,2} - e_3$.
 \item
If $(k_1,k_2) = (1,0)$, then $P$ is a prism unimodularly equivalent to the case $(k_1,k_2) = (0,1)$ and hence to $\maximizer_{5,2}$.
 \item
If $(k_1,k_2) = (1,1)$, then $P = \conv\big((S_1 + S_2 +(-1,1)\big) \times \{-1\} \cup (S_1 + S_2) \times \{1\}) = \maximizer_{6,2} - e_3$. 
 \item
If $(k_1,k_2) = (1,2)$, then $P$ is a prism unimodularly equivalent to the case $(k_1,k_2) = (0,1)$ and hence to $\maximizer_{5,2}$.
\end{compactitem}

This completes the proof that if $P$ is $\Z^3$-maximal, it is unimodularly equivalent to one of the seven polytopes specified in the proposition. 
Conversely, the seven polytopes listed are all $\R^3$-maximal (since each of their facets is blocked) and, in particular, also $\Z^3$-maximal.
\end{proof}

\section{The case of lattice width at least three}\label{sec:lw_three}

In view of Proposition~\ref{prop:lattice_width_two}, to complete the proof of Theorem~\ref{thm:main_three} it remains to show that every $\Z^3$-maximal lattice-free polytope $P \in \cP(\Z^3)$ with lattice width at least three is also $\R^3$-maximal.
Our goal is to derive an algorithm for
a computer search which finds all such polytopes and verifies the latter claim. To be able to do so, we first need reasonably good bounds on the (suitably defined) size of $P$.

\subsection*{Bounds on the size of lattice-free polytopes in $\cP(\Z^3)$ with lattice width at least three}%

The following parameters will be used to quantify the size of a lattice-free polytope $P \in \cP(\Z^3)$. We will use
the volume of both $P$ and its difference body $P-P$ and also the lattice diameter of $P$. Finally, we 
will also use the so-called successive minima: for a convex body $C \subseteq \R^d$ centrally symmetric with respect to the origin, $i \in \{1, \ldots, d\}$,
the minimal value $\lambda > 0$ such that $\lambda C$ contains $i$ linearly independent points of $\Z^d \setminus \{o\}$ is denoted by 
$\lambda_i(C)$ and called the $i$-th {\em successive minimum}\index{minima!successive} of $C$. We remark that clearly, $0 < \lambda_1(C) \le \ldots \le \lambda_d(C)$
and also $\lambda_1(t C) = \frac{1}{t}\lambda_1(C)$ for every positive number $t$. We will quantify the size of $P$ in terms of $\frac{1}{\lambda_1(P-P)}$.

In order to bound the volume from above, we make use of the first and second theorem of Minkowski\index{theorem!of Minkowski}; see \cite[Theorems 4 and 12]{MR1242995},
which link the volume of an $o$-symmetric body to its successive minima. For a convex
body $C \subseteq \R^d$ symmetric with respect to the origin, one has
\begin{align}
	\label{eq:minkowski}
	\vol(C) \le \frac{2^d}{\prod_{i=1}^d \lambda_i(C)} \le \frac{2^d}{\lambda_1(C)^d}.
\end{align}
Consequently, lower bounds on the first successive minimum imply upper bounds on the volume.
Let $K$ be a convex body in $\R^d$, not necessarily a centrally symmetric one. The volume of $K$ and the volume of its difference body $K-K$ are related by the well-known inequality $2^d \vol(K) \le \vol(K-K)$; see for example \cite[Theorem 7.3.1]{MR1216521}. 
Since $K-K$ is $o$-symmetric, one can apply \eqref{eq:minkowski} to obtain
\begin{align}
	\label{eq:gen_minkowski}
	\vol(K) \le \frac{1}{\prod_{i=1}^d \lambda_i(K-K)} \le \frac{1}{\lambda_1(K-K)^d}.
\end{align}
In the following, we will derive a lower bound on $\lambda_1(P-P)$ for a lattice-free polytope $P$ with $\lw{P} \ge 3$.

To establish a connection between the lattice width
of a convex body and the successive minima of its difference body, we use another sequence of parameters connected to a convex body $K$. 
For $i \in \{1,\ldots,d\}$, we denote by $\mu_i(K)$ the \emph{$i$-th covering minimum} of $K$, defined to be the infimum over all $\mu>0$ for which $\mu K + \Z^d$ meets every $(d-i)$-dimensional affine subspace of $\R^d$ (see \cite{MR970611}). 
Clearly, $0 < \mu_1(K) \le \cdots \le \mu_d(K)$. 
Kannan and Lov\'{a}sz \cite[Lemmas~2.3 and 2.5]{MR970611} provided a link between covering minima and successive minima by showing that 
\begin{align} \label{eq:kl}
	\mu_{i+1}(K) \le \mu_i(K) + \lambda_{d-i}(K-K)
\end{align}
holds for every $i \in \{1,\ldots,d-1\}$. Furthermore, they observed $\lw{K} = 1/\mu_1(K)$.
Note that the value $\mu_d(K)$ is the standard {\em covering radius} (also called {\em inhomogeneous minimum}) of $K$ and that $\mu_d(K) \ge 1$ if and only if some translation of $K$ is lattice-free.

Hurkens \cite{MR1060014} proved the following relation between $\mu_2$ and $\mu_1$ for $K \subseteq \R^2$ in the case $d = 2$:
\begin{align}
	\label{eq:hurkens}
	\mu_2(K) \le \left(1 + \frac{2}{\sqrt{3}}\right) \mu_1(K). 
\end{align}
Kannan and Lov\'{a}sz \cite{MR970611} observed that this also implies that \eqref{eq:hurkens} holds for every $d \ge 2$. 
Let us now return to the case $d=3$. 
For a lattice-free convex body $K \subseteq \R^3$, we can therefore combine \eqref{eq:hurkens} with \eqref{eq:kl} for $i=2$ and the fact that $\mu_3(K) \ge 1$ to obtain the following chain of inequalities:
\[
 1 \le \mu_3(K) \le \mu_2(K) + \lambda_1(K-K) \le \left(1 + \frac{2}{\sqrt{3}}\right) \mu_1(K) + \lambda_1(K-K).
\]
Consequently, one has
\begin{align}\label{eq:lambda_vs_lw}
 \lambda_1(K-K) \ge 1 - \left(1 + \frac{2}{\sqrt{3}}\right) \mu_1(K).
\end{align}
Recall that $\mu_1(K)$ is the reciprocal of $\lw{K}$. Hence, the right hand side of \eqref{eq:lambda_vs_lw} is positive whenever $\lw{K} > 1 + \frac{2}{\sqrt{3}}$.
Applying these observations to the case of a lattice-free polytope $P \in \cP(\Z^3)$ with lattice width at least three we obtain a positive lower bound on $\lambda_1(P-P)$ from \eqref{eq:lambda_vs_lw}:
 \begin{align}\label{eq:l1_lower}
  \lambda_1(P-P) \ge \frac{2}{3}\left(1 - \frac{1}{\sqrt{3}}\right) > \frac{1}{4}.
 \end{align}

We are now ready to list the inequalities needed for the proof of Theorem~\ref{thm:main_three}.

\begin{proposition}[Size bounds for lattice-free polytopes in $\cP(\Z^3)$ with lattice width at least three]\label{prop:volume_and_ld}
 Let $P \in \cP(\Z^3)$ be a lattice-free polytope with $\lw{P} \ge 3$. Then, the following statements hold: 
\begin{itemize}
 \item[(a)]\label{item:volume_bound} $\vol(P) \le 27$,
 \item[(b)]\label{item:volume_diff_body} $\vol(P-P) \le 8 \cdot 27$,
 \item[(c)]\label{item:lambda1_bound} $\lambda_1(P-P) > \frac{1}{4}$,
 \item[(d)]\label{item:ld_bound} $\ld{P} \le 3$.
\end{itemize}
\end{proposition}

\begin{proof}[Proof of Proposition~\ref{prop:volume_and_ld}.]
 Let $\lambda_i := \lambda_i(P-P)$ and $\mu_i := \mu_i(P)$ for $i \in \{1,2,3\}$.
 Assertion (c) follows immediately from \eqref{eq:l1_lower}.

 We now prove (a) and (b).
 By \eqref{eq:minkowski} and \eqref{eq:gen_minkowski}, it suffices to prove that the product $\lambda_1 \lambda_2 \lambda_3$ is bounded from below by $1/27$. 
 In the case $\lambda_1 > 1/3$, this is clear, since in this case $\lambda_1 \lambda_2 \lambda_3 \ge \lambda_1^3 > 1/27$.
 We switch to the case $\lambda_1 \le 1/3$. Applying \eqref{eq:kl} for $i=1$, we
 get $\lambda_2 \ge \mu_2 - \mu_1$ and applying \eqref{eq:kl} for $i=2$ yields $\mu_2 \ge \mu_3 - \lambda_1$. Combining these two inequalities, we get
 $\lambda_2 \ge \mu_3 - \lambda_1 - \mu_1$.
 Since $P$ is lattice-free, we have $\mu_3 \ge 1$. We further have $\mu_1 \le 1/3$ because $1 / \mu_1 = \lw{P} \ge 3$.
 Hence, $\lambda_2 \ge 2/3 - \lambda_1$ and thus, $\lambda_1 \lambda_2 \lambda_3 \ge \lambda_1 (2/3 - \lambda_1)^2$.
Taking into account (c), we have $1/4 < \lambda_1 \le 1/3$. For a moment, we view $\lambda_1$ as a variable and determine a lower bound of the function
 $f(\lambda_1) := \lambda_1 (2/3 - \lambda_1)^2$ 
on the interval $\left(\frac{1}{4}, \frac{1}{3}\right]$.
 The latter function is decreasing on this interval and hence attains its minimum on this interval for $\lambda_1 = 1/3$. This
 yields $\lambda_1 \lambda_2 \lambda_3 \ge f(\lambda_1) \ge f(1/3) = 1/27$ and completes the proof of (a) and (b). 

This leaves us to show (d).
We observe the following connection between lattice diameter and first successive minimum. 
By the definition of $\ld{P}$, there exist distinct points $z, z' \in P \cap \Z^3$ with $u := \frac{1}{\ld{P}}(z - z') \in \Z^3$. Thus,
$u$ is a non-zero integral vector in $\frac{1}{\ld{P}}(P-P)$. Taking into account the definition of 
the first successive minimum, we get
\begin{align*}
 \lambda_1 \le \frac{1}{\ld{P}}. 
\end{align*}
Hence, we have $\ld{P} \le 1 / \lambda_1 < 4$ in view of (c) and since $\ld{P} \in \N$, this shows (d).
\end{proof}

\subsection*{Completing the case of lattice width at least three by a computer search}\label{sec:computer_search}%
Based on the bounds on the size of lattice-free polytopes in $\cP(\Z^3)$ with lattice width at least three, we
can now derive a search algorithm that finds, up to affine unimodular transformations, all
polytopes in the set
\begin{equation}
    \label{computerSearchGoal}
    \setcond{P \in \cP(\Z^3)}{P \text{ is } \Z^3\text{-maximal lattice-free and } \lw{P} \ge 3}.
\end{equation}
We present the algorithm along with mathematical arguments that verify its correctness.
We give a rather high-level description omitting algorithmic details that can be resolved in a straightforward manner.
We implemented our search algorithm in Python~2.7 -- a high-level language with a self-explanatory syntax.
The code of this implementation can be found in the appendix.
Our implementation contains a straightforward test that checks for each polytope~$ P $ in~\eqref{computerSearchGoal}
whether it is also $ \R^3 $-maximal:
we check whether the relative interior of each facet of~$ P $ contains an integral point.
As a result, an execution of our code confirms (within less than half an hour) that every polytope
in~\eqref{computerSearchGoal} is indeed $ \R^3 $-maximal.

Since all $ \R^3 $-maximal lattice-free polytopes in~$ \cP(\Z^3) $ were enumerated in~\cite{MR2855866}, it was not our intention
to output an irredundant list of all polytopes in~\eqref{computerSearchGoal}.
In fact, our algorithm enumerates many polytopes that are pairwise unimodularly equivalent. We decided to refrain from excluding such redundancy
to keep the code as simple as possible and thus more easily verifiable.

\subsubsection*{Step 1: Fixing a boundary fragment}
Let $P \in \cP(\Z^3)$ be $\Z^3$-maximal lattice-free and let $\fixedLatticeDiameter := \ld{P}$.
In view of Proposition~\ref{prop:volume_and_ld}, the possible values of~$ \fixedLatticeDiameter $
are~$ 1 $, $ 2 $, and~$ 3 $.
Because~$ P $ is lattice-free, there exists a facet~$ F $ of~$ P $ with~$ \ld{F} = \ld{P} = \fixedLatticeDiameter $ and by applying
an appropriate unimodular transformation, we can assume $F \subseteq \R^2 \times \{0\}$.
Let us first treat the case that~$ \fixedLatticeDiameter = 1 $.
Howe's Theorem (see, e.g., \cite[Theorem 1.5]{MR798388}) asserts that if $P \in \cP(\Z^3)$ satisfies $P \cap \Z^3 = \vertset{P}$, then 
the lattice width of $P$ is one. 
Since, up to unimodular equivalence, $\R^2 \times [0,1]$ is the only $\Z^2$-maximal lattice-free polyhedron with lattice width one, such $P$ is not
$\Z^3$-maximal.
Hence, $P$ has to contain an integral point which is not a vertex of $P$, and since $P$ is lattice-free, this point has to be in the boundary of $P$.
Clearly, there cannot be an integral point in the relative interior of an edge of
$P$, since otherwise this edge has at least three collinear integral points, contradicting $\fixedLatticeDiameter = 1$. 
This implies that we can choose $F$ such that $\relintr{F} \cap \Z^3 \neq \emptyset$. Suppose $|F \cap \Z^3| > 4$. Then two points in $|F \cap \Z^3|$, say $x$ and $y$, are congruent modulo two and hence, the line segment $[x,y]$ contains $\frac{1}{2}(x+y) \in \Z^3$, a contradiction
to $\fixedLatticeDiameter = 1$. Thus, $F$ is a triangle with precisely one integral point in its relative interior. It follows that $F$
is unimodularly equivalent to $\conv\big((-1,-1,0), \, (1,0,0), \, (0,1,0) \big)$; see \cite{MR1039134}, where all integral polygons with precisely one interior
integral point were listed.

It is easy to check that in the the case~$ \fixedLatticeDiameter = 2 $, up to unimodular transformation, ~$ F $ contains the triangle $\conv\big((0,0,0), \, (2,0,0), \, (0,1,0)\big)$
and likewise, in the case~$ \fixedLatticeDiameter = 3 $, up to unimodular
transformation, $F$ contains the triangle $\conv\big((0,0,0), \, (3,0,0), \, (0,1,0)\big)$. 
Thus, we may assume that~$ P $ contains one of the following three triangles~$ \fixedBase $ in its boundary:
\begin{equation}
    \label{eqFixedBase}
    \fixedBase :=
    \begin{cases}
        \conv\big((-1,-1,0), \, (1,0,0), \, (0,1,0)\big) & \text{if } \fixedLatticeDiameter = 1, \\
        \conv\big((0,0,0), \, (2,0,0), \, (0,1,0)\big) & \text{if } \fixedLatticeDiameter = 2, \\
        \conv\big((0,0,0), \, (3,0,0), \, (0,1,0)\big) & \text{if } \fixedLatticeDiameter = 3.
    \end{cases}
\end{equation}

\subsubsection*{Step 2: Inscribing a pyramid}
Clearly, after fixing~$\ell$ and the corresponding $ \fixedBase $, without loss of generality we can assume that we have~$ x_3 \ge 0 $ for all points~$
(x_1,x_2,x_3) \in P $.
Let~$h := \max \setcond{x_3}{(x_1,x_2,x_3) \in P}$ and fix some vertex of $P$ of the form $ \fixedApex = (a_1,a_2,h) $.
By the definition of the lattice width, we have $h \ge \lw{P}$ and since we assume $\lw{P} \ge 3$, we have $h \ge 3$.
By applying an appropriate unimodular transformation that keeps the third component unchanged, we may assume that~$ 0 \le a_1,a_2 \le
h - 1 $ holds. To see that, observe that for every pair of integers $k,k'$ the linear mapping $\varphi$\label{page:ut_a_h} given
by $\varphi(e_1) = e_1$, $\varphi(e_2) = e_2$ and $\varphi(e_3) = k e_1 + k' e_2 + e_3$ is unimodular and maps $(a_1,a_2,h)$ to $(a_1 + k h, a_2 + k' h, h)$.
In order to obtain an upper bound on~$ h $, we define the tetrahedron~$ \fixedPyramid := \conv(\fixedBase \cup \{\fixedApex\}) $
and observe that we have
\[
    \tbinom{6}{3} \cdotp \vol(\fixedPyramid)
    = \vol(\fixedPyramid - \fixedPyramid)
    \le \vol(P-P)
    \le 8 \cdotp 27,
\]
where the first equality follows from \cite[Theorem 2]{MR0092172} and the last inequality follows from
Proposition~\ref{prop:volume_and_ld}.
Denoting by~$ \operatorname{area}(\fixedBase) $ the area of~$ \fixedBase $ and using $ \vol(\fixedPyramid) = \tfrac{1}{3} \cdotp h \cdotp
\operatorname{area}(\fixedBase) $, we thus obtain
\[
    h \le \left\lfloor \frac{3 \cdotp 8 \cdotp 27}{\tbinom{6}{3} \cdotp \operatorname{area}(\fixedBase)} \right\rfloor.
\]
In the cases~$ \fixedLatticeDiameter = 2 $ and~$ \fixedLatticeDiameter = 3 $, this bound evaluates to~$ h \le 32 $ and~$
h \le 21 $, respectively.
However, in the case~$ \fixedLatticeDiameter = 1 $, the following Lemma provides a tighter upper bound on~$ h $.
\begin{lemma}
    Let~$ P \in \cP(\Z^3) $ be a lattice-free polytope such that~$ P $ contains~$ (a_1,a_2,h) \in \Z^3 $ and also the points
    $ (-1, -1, 0) $, $ (1, 0, 0) $ and $ (0, 1, 0) $.
    Then,~$ h \le 12 $.
\end{lemma}

\begin{proof}
We argue by contradiction by assuming $h > 12$ and showing that this contradicts the lattice-freeness of $P$. Let $D$ denote the triangle $\conv\big((-1,-1,0),(1,0,0),(0,1,0)\big)$ and consider $D' := D \cap (-D)$. Observe that $o \in D$ and hence,
$o \in D'$. Furthermore, $D'$ is symmetric with respect to $o$ by construction. We also introduce the double pyramid $ R := \conv(D' \cup \{\pm (a_1,a_2,h)\}) $. One can verify directly that the area of $D'$ is one and thus, the volume of $R$ is $\frac{2}{3}h > 8$. 
Since $R$ is an $o$-symmetric convex body, we can apply \eqref{eq:minkowski} to obtain
$\frac{2^3}{\lambda_1(R)^3} \ge \vol(R) > 2^3$ or, equivalently, $\lambda_1(R) < 1$. By definition of the successive minima, this implies that
$\intr{R}$ contains a point $z$ of $\Z^3 \setminus \{o\}$. Because of the symmetry of $R$, this also implies $-z \in \intr{R}$. As $D' \cap \Z^3 \setminus \{o\} = \emptyset$,
either $z$ or $-z$ has to be contained in $\intr{R} \cap \big( \R^2 \times (0,\infty) \big)$. The latter set, however, is contained in $\intr{P}$,
a contradiction to the lattice-freeness of $P$.
\end{proof}

We now summarize the previous observations as follows: we may assume that~$ P $ is contained in $\R^2 \times [0,\bar{h}]$,
where
\[
    \bar{h} := \begin{cases}
        12 & \text{if } \fixedLatticeDiameter = 1, \\
        21 & \text{if } \fixedLatticeDiameter = 2, \\
        32 & \text{if } \fixedLatticeDiameter = 3,
    \end{cases}
\]
and that $P$
 contains the tetrahedron~$ \fixedPyramid = \conv(\fixedBase \cup \{a\})$, where $B$ is defined as in \eqref{eqFixedBase} and
\begin{align}\label{eq:set_for_a}
   a \in S_{\bar{h}} := \setcond{(x_1, x_2, x_3) \in \Z^3}{0 \le x_1,x_2 \le x_3 - 1, \, 0 \le x_3 \le \bar{h}}.
\end{align}

We now enumerate all elements of the set in \eqref{eq:set_for_a}.
Without further qualifications, this enumeration yields several ten thousand points.
However, recall that~$ P $ is lattice-free and satisfies
$ \lambda_1(P-P) > \frac{1}{4} $ by Proposition~\ref{prop:volume_and_ld}.
Since $\fixedPyramid \subseteq P$, these two properties also have to be satisfied by~$ \fixedPyramid $.
Furthermore, because $P$ has lattice diameter $\fixedLatticeDiameter$ and $\fixedPyramid \subseteq P$, we have $\ld{\fixedPyramid} \le \fixedLatticeDiameter$, while
on the other hand $\fixedBase \subseteq \fixedPyramid$ yields $\ld{\fixedPyramid} \ge \ld{\fixedBase} = \fixedLatticeDiameter$.
Hence, we obtain that $T$ is lattice-free and satisfies $\ld{T} = \fixedLatticeDiameter$ as well $\lambda_1(T-T) > \frac{1}{4}$.
The computer search checks for each of the possible choices for~$ \fixedApex = (a_1,a_2,h) $
whether these three properties are fulfilled for the tetrahedron $\conv(\fixedBase \cup \{(a_1,a_2,h)\})$. 
As a consequence, we end up with only $ 69 $ valid choices for~$ \fixedApex $ in total.

\subsubsection*{Step 3: Collecting all candidate vertices}
In the next step, for a given combination of fixed lattice diameter~$ \fixedLatticeDiameter $ and pyramid~$ \fixedPyramid $ (given by $\fixedBase$
and $a = (a_1,a_2,h)$ as introduced above), we enumerate all lattice-free polytopes $P$ in $\cP(\Z^3)$ which have 
lattice diameter $\fixedLatticeDiameter$ and contain $\fixedPyramid$. 
In order to do so, we
enumerate a finite subset of $\Z^3$, depending only on $\fixedPyramid$, which we denote by
$\candidates{\fixedPyramid}$ (we call the elements of $\candidates{\fixedPyramid}$ {\em candidate vertices}) 
such that for every $P$ as above one has $ \vertset{P} \subseteq \candidates{\fixedPyramid}$.
This is done as follows.
First, recall that every vertex~$ v = (v_1,v_2,v_3) $ of~$ P $ satisfies~$ 0 \le v_3 \le h $ by our choice of $h$.
By Proposition~\ref{prop:volume_and_ld}, we further have
\[
    \vol\big(\conv(\fixedPyramid \cup \{v\})\big) \le \vol(P) \le 27.
\]
The following lemma shows that all points $x \in \R^3$ satisfying $ \vol\big(\conv(\fixedPyramid \cup \{x\})\big) \le 27 $ lie in an appropriate
homothetic copy of $\fixedPyramid$.
\begin{lemma}
    Let~$ T \subseteq \R^3 $ be a tetrahedron, $c$ the barycenter of its vertices and let $ t \ge \vol(T) $. Then
    \[
        \Big\{x \in \R^3 \; : \; \vol\big(\conv(T \cup \{x\})\big) \le t\Big\} \subseteq \lambda \cdotp T + (1-\lambda) \cdotp c,
    \]
    where~$ \lambda := 4 \cdotp \left( \tfrac{t}{\vol(T)} - 1 \right) + 1 $.
\end{lemma}

\begin{proof}
Consider a point $x \in \R^3$ not belonging to $\lambda T + (1 - \lambda) c$. We show that then
one has $\vol\big(\conv(T \cup \{x\})\big) > t$. Let $\beta_1, \ldots, \beta_4$ be the barycentric coordinates of
$x$ with respect to the vertices of $\lambda T + (1 - \lambda) c$. It is straightforward to verify that the barycentric coordinates
of $x$ with respect to the vertices of $T$ are $\beta_1 + \frac{1 - \lambda}{4}, \ldots, \beta_4 + \frac{1 - \lambda}{4}$. Since
$x \not\in \lambda T + (1 - \lambda) c$, at least one of $\beta_1, \ldots, \beta_4$, say $\beta_1$, is negative and hence, the corresponding barycentric
coordinate of $x$ with respect to $T$ is $\gamma_1 := \beta_1 + \frac{1 - \lambda}{4} < \frac{1 - \lambda}{4} = - \frac{t}{\vol(T)} + 1 < 0$.

Let $v_1$ be the vertex of $T$ associated with this barycentric coordinate and let $F := \conv(\vertset{T} \setminus \{v_1\})$.
Let $d_{v_1}$ denote the distance of $v_1$ from $H := \aff(\vertset{T} \setminus \{v_1\})$ and let $d_x$ denote the distance of $x$ from $H$.
Then from the geometric interpretation of
the barycentric coordinates, one has $d_x = |\gamma_1|d_{v_1} = - \gamma_1 d_{v_1}$. 
Likewise, it follows from the geometric interpretation and
the fact that $\gamma_1 < 0$ that $T$ and $\conv(F \cup \{x\})$ do not have common interior points.
In view of $- \gamma_1 < 1 - \frac{t}{\vol(T)}$, this yields
the following chain of inequalities:
\begin{align*}
\vol\big(\conv(T \cup \{x\})\big) 
&\ge \vol(T) + \vol\big(\conv(F \cup \{x\})\big)\\ &= \vol(T) - \gamma_1\vol\big(\conv(F \cup \{v_1\})\big)\\ 
& = (1 - \gamma_1)\vol(T) > \vol(T) \big(1 + \frac{t}{\vol(T)} - 1\big) = t,
\end{align*}
which proves the statement.
\end{proof}
Thus, if~$ c $ denotes the barycenter of the vertices of~$ \fixedPyramid $, each vertex of~$ P $ is contained in 
\[
    \safeRegion{\fixedPyramid} := \setcond{(x_1,x_2,x_3) \in \big( \lambda \cdotp \fixedPyramid + (1-\lambda) \cdotp
    c \big) \cap \Z^3}{0 \le x_3 \le h },
\]
where~$ \lambda := 4 \cdotp \left( \tfrac{27}{\vol(\fixedPyramid)} - 1 \right) + 1 $.
We call $\safeRegion{\fixedPyramid}$ the {\em search region} for $T$.
Similarly to the previous step, for each vertex~$ v $ of~$ P $ we have that~$ \conv(\fixedPyramid \cup \{v\}) $ is
lattice-free and its lattice diameter is bounded by~$ \fixedLatticeDiameter $.
Hence, a valid choice for the set~$ \candidates{\fixedPyramid} $ is
\begin{align*}
    \candidates{\fixedPyramid} := \{ v \in \safeRegion{\fixedPyramid} \setcondsep
        \ & \conv(\fixedPyramid \cup \{v\}) \text{ is lattice free}, \\
        & \ld{\conv(\fixedPyramid \cup \{v\})} = \fixedLatticeDiameter \}.
\end{align*}
For the enumeration of $\candidates{\fixedPyramid}$, we enumerate $\safeRegion{\fixedPyramid}$ and check the defining conditions for $\candidates{\fixedPyramid}$.
Since for most choices of $\fixedPyramid$, the set $\safeRegion{\fixedPyramid}$ is huge, this is the most time-consuming step. Therefore, in order to reduce the
running time, we define $\candidates{\fixedPyramid}$ without use of the successive minima. At the end of Step 3, $\candidates{\fixedPyramid}$ is stored
as a list of points and available for further computations.

\subsubsection*{Step 4: Combining the candidate vertices}
In the final step, for each fixed~$ \fixedLatticeDiameter $ and~$ \fixedPyramid $ it remains to enumerate all lattice-free polytopes
with lattice diameter $\ell$ and lattice width at least three which can be obtained by taking the convex hull of $T$ and a set of
elements of $\candidates{\fixedPyramid}$, i.e., all polytopes
in the set
\[
    \computerFinalSet := \setcond{P \in \cP(\Z^3)}{\vertset{P} \subseteq \candidates{\fixedPyramid}
, \, T \subseteq P, \, P \text{ lattice-free}, \, \ld{P} =
\fixedLatticeDiameter, \, \lw{P} \ge 3}.
\]
This is done in the following way.
Let~$ v_1,\dotsc,v_k $ be the points of~$ \candidates{\fixedPyramid}$.
We construct $\computerFinalSet$ iteratively by listing the elements of sets $\computerIntermediateSet{0}, \ldots, \computerIntermediateSet{k}$, where 
$\computerIntermediateSet{0} = \{T\}$ and $\computerIntermediateSet{k} = \computerFinalSet$. The set $\computerIntermediateSet{i}$ is constructed
from $\computerIntermediateSet{i-1}$ in a way that for every $i \in \{0, \ldots, k\}$, $\computerIntermediateSet{i}$ is the set of all polytopes
in $\cP(\Z^3)$ satisfying the following conditions:
\begin{align}
 \vertset{P} \subseteq \vertset{T} \cup \{v_1,\dotsc,v_i\}, \label{eq:set_i_first}\\
 T \subseteq P, \, P \text{ is lattice-free}, \, \ld{P} = \fixedLatticeDiameter, \label{eq:set_i_second}\\
 \lw{P \cup \{v_{i+1},\dotsc,v_k\}} \ge 3. \label{eq:set_i_third}
\end{align}
Obviously one has~$ \computerIntermediateSet{0} := \{T\} $ and can
compute $\computerIntermediateSet{i}$ as the union of the set of all polytopes of $\computerIntermediateSet{i-1}$ 
satisfying \eqref {eq:set_i_third} and the set of all polytopes $P = \conv(P' \cup \{v_i\})$, where $P' \in \computerIntermediateSet{i-1}$, satisfying
conditions \eqref{eq:set_i_second} and \eqref{eq:set_i_third}.
Only the third condition needs explaining: this is to ensure that for our preliminary polytope `built' from $\fixedPyramid$ and the first $i$ elements of  $\candidates{\fixedPyramid}$, which might not have lattice width three, it is still possible to obtain a polytope of lattice width at least three by `adding' points of $\candidates{\fixedPyramid}$ which have not yet been considered. If this is not possible, there is no need to consider this polytope any further.
Observe that we have~$ \computerFinalSet = \computerIntermediateSet{k} $ and that each~$
\computerIntermediateSet{i} $ can be recursively computed for $i \in \{1, \ldots, k\}$. 

With an arbitrary choice of the sequence~$ v_1, \ldots, v_k $, the cardinalities of some of the sets $ \computerIntermediateSet{1}, \ldots, \computerIntermediateSet{k} $
can be so large that a computation becomes practically impossible.
However, this can be avoided by sorting the~$ v_j $'s decreasingly according to the absolute value of their second
coordinates.
The idea behind this order is that, for most~$ \fixedPyramid $, the majority of points in~$ \candidates{\fixedPyramid} $
(including the vertices of $ \fixedPyramid $) have a second coordinate with small absolute value.
Thus, in order to arrive at a polytope with lattice width at least three, at least one of the~$ v_j $'s appearing early in the sorted list
has to belong to~$ P $.
Due to condition~\eqref{eq:set_i_third}, any polytope not containing one of the first~$ v_j $'s is
then discarded from the sets~$ \computerIntermediateSet{i} $ for already small~$ i $.
Furthermore, for any computed polytope containing one of the~$ v_j $'s from the beginning of the sorted list, the number of further candidates
that can be added to arrive at a lattice-free polytope turns out to be very small.

In the described procedure, for all constructed polytopes we need to decide whether their lattice width is at least
three.
This is done by simply computing the width with respect to a few explicit directions; more specifically, the directions in the set $V$ defined in
the following lemma.

\begin{lemma}
    Let~$ P \subseteq \R^3 $ be a polytope containing~$ o, e_1, e_2 $ and some point~$ a = (a_1, a_2, h) \in \Z^3 $ with
    $ h \ge 1 $.
    Then~$ \lw{P} \ge 3 $ holds if and only if
    \begin{equation}
        \label{eqCondLatticeWidthAtLeastThree}
        \width{P,v} \ge 3
    \end{equation}
    holds for all~$ v $ in the set
    \begin{align*}
        V := \big\{(v_1, v_2, v_3) \in \Z^3 \setminus \{o\} \setcondsep \ & |v_1|, |v_2|, |v_1 - v_2| \le 2, \\
        & \tfrac{1}{h}(-2 - v_1 a_1 - v_2 a_2) \le v_3 \le \tfrac{1}{h}(2 - v_1 a_1 - v_2 a_2) \big\}.
    \end{align*}
\end{lemma}

\begin{proof}
    If $\lw{P} \ge 3$, then by definition of the lattice width, \eqref{eqCondLatticeWidthAtLeastThree} holds for every
    $v \in  \Z^3 \setminus \{o\}$ and hence in particular for all $v \in V$.
    To show the reverse implication, it suffices to show that if~$ \lw{P} \le 2 $ holds, then there exists some~$ v \in
    V $ violating~\eqref{eqCondLatticeWidthAtLeastThree}.
    Suppose we have~$ \lw{P} \le 2 $.
    Then there exists a (primitive) vector~$ v^* = (v_1^*, v_2^*, v_3^*) \in \Z^3 \setminus \{ o \} $ with
    \begin{align*}
	2 \ge \width{P,v^*} \ge
\width{\{o, e_1, e_2\},v^*}
        = \max \{ |v^*_1|, \, |v^*_2|, \, |v^*_1 - v^*_2| \},
    \end{align*}
    where the second inequality follows from the fact that we have~$ o,e_1,e_2 \in P $.
    Finally, since we have~$ o,a \in P $, we must have
    \[
        | \langle v^*,a \rangle - \langle v^*, o \rangle | \le 2,
    \]
    which is equivalent to
    \[
        -2 \le v_1^* a_1 + v_2^* a_2 + v_3^* h \le 2,
    \]
    as claimed.
\end{proof}

Finally, having computed the set~$ \computerIntermediateSet{k} $, for each polytope~$ P \in
\computerIntermediateSet{k} $ we heuristically search for a certificate showing that $P$ is not $\Z^3$-maximal.
This is done by testing for a few points~$ p \in \Z^3 \setminus P $ in a certain neighborhood of $P$ if $ \conv(P \cup \{p\}) $ is still
lattice-free.
If this check fails for all considered choices of~$ p $, the polytope $ P $ is regarded as being potentially~$ \Z^3 $-maximal and we then
test whether~$ P $ is even $ \R^3 $-maximal.
It turns out that~$ P $ is indeed $ \R^3 $-maximal in all these cases and hence, we obtain the following result.

\begin{proposition}\label{prop:lattice_width_larger}
 Let $P \in \cP(\Z^3)$ be a lattice-free polytope with $\lw{P} \ge 3$. Then, $P$ is $\Z^3$-maximal if and only if $P$ is $\R^3$-maximal.
\end{proposition}

Combining this with the results of the previous section, we obtain the proof of Theorem~\ref{thm:main_three}.

\begin{proof}[Proof of Theorem~\ref{thm:main_three}.]
 Let $P \in \cP(\Z^3)$ be lattice-free. If $P$ is $\R^3$-maximal, then it is also $\Z^3$-maximal by definition of $\Z^3$-maximality.
Conversely, if $P$ is $\Z^3$-maximal and bounded, then in
view of Propositions~\ref{prop:lattice_width_two} and \ref{prop:lattice_width_larger} we have that $P$ is also $\R^3$-maximal.
If $P$ is $\Z^3$-maximal and unbounded, then it is unimodularly equivalent 
to $\conv(o, 2e_1, 2e_2) \times \R$ or to
$[0,1] \times \R^2$; see \cite[Proposition 3.1]{MR2855866}. In both cases,
we have that P is also $\R^3$-maximal.
\end{proof}

\begin{remark}\label{rem:two}
Consider again the question asked in Remark~\ref{rem:one}: for which pairs $s,d \in \N$ are $\Z^d$-maximality and $\R^d$-maximality equivalent
for lattice-free polyhedra in $\cP(\frac{1}{s}\Z^d)$?
Theorem~\ref{thm:main_three} shows that this equivalence does hold for $d = 3$ and $s = 1$. 
In Remark~\ref{rem:one}, we stated that equivalence also holds for $d=2$ and $s \in \{1,2\}$,
while it trivially also holds for $d = 1$.
For all other pairs of $d,s \in \N$, this equivalence does not hold: polytopes in $\cP(\frac{1}{s}\Z^d)$
which are $\Z^d$-maximal but not $\R^d$-maximal can be constructed following the ideas of \cite[Proof of Theorem 3.2]{MR2832401}.
\end{remark}

\bibliographystyle{amsalpha}
\bibliography{references}

\section*{Appendix}

This appendix contains an implementation of the algorithm presented in Section~\ref{sec:lw_three}
in Python 2.7.
Executing this code yields Proposition~\ref{prop:lattice_width_larger}.
It is divided into three files, the first of
which contains the main part of the algorithm, while the other two provide basic functions for computations on three-dimensional polytopes
and basic functions and objects needed for geometric computations.

\lstset{language=Python, 
        basicstyle=\ttm, 
        keywordstyle=\ttb,
        commentstyle=\tti,
        stringstyle=\color{gray},
	breakatwhitespace=true,
breaklines=true,
  numbers=left,                    
  numbersep=5pt,                   
  numberstyle=\tiny\color{gray}, 
        showstringspaces=false,
rulecolor=\color{black}, 
        identifierstyle=\ttm,
        procnamekeys={def,class}}

\subsection*{search.py}

\begin{lstlisting}
"""
This is the main part of the code. The procedure is executed by calling the function "main".
"""

from basics import Vector3d
from fractions import Fraction
from polytopes import Polytope3d

import basics
import math

UPPER_BOUND_VOLUME = 27
STRICT_LOWER_BOUND_SUCC_MIN = Fraction(1, 4)

def main():
    """Executes the search algorithm."""
    # Step 1: enumerate boundary fragments (and run search)
    print "case: lattice diameter = 1"
    run_search(base = [Vector3d(x, y, 0) for (x, y) in [(1, 1), (-1, 0), (0, -1)]],
               upper_bound_height = 12,
               lattice_diameter = 1)

    print "case: lattice diameter = 2"
    run_search(base = [Vector3d(x, y, 0) for (x, y) in [(0, 0), (2, 0), (0, 1)]],
               upper_bound_height = 32,
               lattice_diameter = 2)

    print "case: lattice diameter = 3"
    run_search(base = [Vector3d(x, y, 0) for (x, y) in [(0, 0), (3, 0), (0, 1)]],
               upper_bound_height = 21,
               lattice_diameter = 3)

# -------------------------------------------------------------------------

def run_search(base, upper_bound_height, lattice_diameter):
    # Step 2: enumerate apices
    print "  - computing apices ..."
    apices = list(enumerate_apices(base, upper_bound_height, lattice_diameter))
    print "  - apices:", len(apices)

    # Step 3: collect further candidate vertices
    for apex in apices:
        print "  - investigating pyramid with base", base, "and apex", apex
        print "    - computing candidates ..."

        pyramid = Polytope3d(base + [apex])
        candidates = get_sorted_candidates(pyramid, lattice_diameter)
        print "    - candidates:", len(candidates)

        # Step 4: enumerate combinations
        # - searches for a 'critical' polytope (see docstring of 'search_critical_polytope' for a definition) with base
        #   'base', apex 'apex', and lattice diameter 'lattice_diameter'
        # - as a result, no such critical polytope exists, which verifies the claim
        Q = search_critical_polytope(pyramid, apex, lattice_diameter, candidates)
        if Q:
            print "! error, found critical polytope", Q.vertices()
            exit(1)

def enumerate_apices(base, upper_bound_height, lattice_diameter):
    """
    Returns a list of all possible apices for the given boundary fragment 'base', see explanation of Step 2.
    """
    for z in range(3, upper_bound_height + 1):
        for x in range(z):
            for y in range(z):
                apex = Vector3d(x, y, z)
                pyramid = Polytope3d(base + [apex])

                if is_valid_pyramid(pyramid, lattice_diameter):
                    yield apex

def is_valid_pyramid(pyramid, lattice_diameter):
    """
    Returns true iff 'pyramid' is lattice-free, has lattice diameter of at most 'lattice_diameter' and has first
    successive minimum less than 1/4.
    """
    if not pyramid.is_lattice_free():
        return False

    if pyramid.has_ld_greater_than(lattice_diameter):
        return False

    if pyramid.successive_minimum_at_most(STRICT_LOWER_BOUND_SUCC_MIN):
        return False

    return True

def get_sorted_candidates(pyramid, lattice_diameter):
    """
    Computes a sorted list of candidates for given pyramid 'pyramid' and lattice diameter 'lattice_diameter' (see the
    explanations of Steps 3 and 4 as well as the remarks on the order of the list).
    """
    candidates = []

    for c in integer_points_in_search_region(pyramid):
        if not pyramid.contains(c):
            Q = Polytope3d(pyramid.vertices() + [c])
            if Q.is_lattice_free() and not Q.has_ld_greater_than(lattice_diameter):
                candidates.append(c)

    is_in_strip = lambda c: pyramid.ymin() <= c.y and c.y <= pyramid.ymax()

    return [c for c in candidates if is_in_strip(c)] + [c for c in candidates if not is_in_strip(c)]

def integer_points_in_search_region(pyramid):
    """
    Computes a search region for integer points that can be added to 'pyramid' without exceeding the volume bound
    'volume_bound' (see explanation of Step 3).
    """
    vol = basics.volume_of_simplex(pyramid.vertices())
    assert vol <= UPPER_BOUND_VOLUME

    vertices = pyramid.vertices()
    center = (vertices[0] + vertices[1] + vertices[2] + vertices[3]) * Fraction(1,4)
    scaling = 1 + 4 * (Fraction(UPPER_BOUND_VOLUME) / Fraction(vol) - 1)
    search_region = pyramid.homo_copy(scaling, center)

    return search_region.integer_points(zmin = 0, zmax = pyramid.zmax())

def search_critical_polytope(pyramid, apex, lattice_diameter, candidates):
    """
    This function computes all lattice-free polytopes with lattice diameter <= 'lattice_diameter' and lattice width >= 3
    that can be constructed from 'pyramid' by adding points of 'candidates'. For each such polytope, we check whether it
    is 'critical', i.e., potentially Z^3-maximal lattice-free but not R^3-maximal lattice-free. The method of generating
    such a list is described in the explanation of Step 4.
    """
    if is_critical(pyramid, apex):
        return pyramid

    polytopes_built = set([pyramid])
    num_candidates = len(candidates)

    while candidates:
        c = candidates.pop()
        polytopes_new = set(polytopes_built)

        # For each polytope P, keep P and add P' := P \cup \{ c \} if P' is lattice free and ld(P') <= lattice_diameter
        for P in polytopes_built:
            if not P.contains(c):

                Q = Polytope3d(P.vertices() + [c])
                if Q.is_lattice_free() and not Q.has_ld_greater_than(lattice_diameter):
                    if is_critical(Q, apex):
                        return Q
                    polytopes_new.add(Q)

        # Only keep polytopes that, by adding the remaining candidates, can be extended to ones with lattice width >= 3
        polytopes_built = set([P for P in polytopes_new if
                               not has_lattice_width_at_most_two(P.vertices() + candidates, apex)])
        print "    - completed %d / %d (polytopes stored: %d)" % \
            (num_candidates - len(candidates), num_candidates, len(polytopes_built))

        if not polytopes_built:
            return

def is_critical(P, apex):
    """
    Returns true iff P
    - has lattice width at least three,
    - is not R^3-maximal lattice-free, but
    - seems to be Z^3-maximal lattice-free.
    """
    if has_lattice_width_at_most_two(P.vertices(), apex):
        return False
    if P.all_facets_blocked():
        return False

    # Heuristically try to add some integer points outside of P
    (min_x, min_y, min_z, max_x, max_y, max_z) = P.bounding_box()
    return not P.is_lattice_free_extendable([Vector3d(x, y, z) for x in range(min_x - 1, max_x + 2) \
                                                               for y in range(min_y - 1, max_y + 2) \
                                                               for z in range(min_z - 1, max_z + 2)])

def has_lattice_width_at_most_two(points, apex):
    """
    Returns true iff 'points' has lattice width of at most two, see the lemma in explanation of Step 4.
    """
    for vx in range(-2, 3):
        for vy in range(-2, 3):
            if abs(vx - vy) <= 2:
                min_z = basics.custom_ceil(Fraction(- 2 - vx * apex.x - vy * apex.y, apex.z))
                max_z = basics.custom_floor(Fraction(2 - vx * apex.x - vy * apex.y, apex.z))

                for vz in range(min_z, max_z + 1):
                    if vx != 0 or vy != 0 or vz != 0:
                        v = Vector3d(vx, vy, vz)

                        if max([p * v for p in points]) - min([p * v for p in points]) <= 2:
                            return True
    return False

# -------------------------------------------------------------------------

main()
\end{lstlisting}

\subsection*{polytopes.py}

\begin{lstlisting}
"""Straightforward implementation of basic functions on three-dimensional (rational) polytopes."""

import basics
import itertools

from basics import Vector3d
from fractions import Fraction

class Inequality3d(object):
    """Represents a linear inequality of the form a * x <= b."""

    def __init__(self, lhs, rhs):
        """Creates a linear inequality of the form a * x <= b, where a = 'lhs' and b = 'rhs'."""
        self.lhs = lhs
        self.rhs = rhs
        if self.lhs.is_zero():
            raise Exception("trivial inequality: %s" % str(self))

    def __str__(self):
        return str(self.lhs) + " <= " + str(self.rhs)

    def __repr__(self):
        return str(self)

    def __eq__(self, other):
        """Returns true iff all coefficients including the right-hand side of 'self' and 'other' coincide."""
        return self.lhs == other.lhs and self.rhs == other.rhs

    def __ne__(self, other):
        return self.lhs != other.lhs or self.rhs != other.rhs

    def __hash__(self):
        return hash(str(self))

    def is_satisfied(self, p):
        """Returns true iff 'p' satisfies the inequality."""
        return self.lhs * p <= self.rhs

    def is_strictly_satisfied(self, p):
        """Returns true iff the inequality is strict for 'p'."""
        return self.lhs * p < self.rhs

    def holds_with_equality(self, p):
        """Returns true iff 'p' satisfies the inequality with equality."""
        return self.lhs * p == self.rhs

# -------------------------------------------------------------------------

class Polytope3d(object):
    """Represents a three-dimensional (rational) polytope."""

    def __init__(self, points):
        """Constructs the convex hull of 'points'."""
        assert(len(points) >= 3)

        # store bounding box
        self._xmin = min([p.x for p in points])
        self._ymin = min([p.y for p in points])
        self._zmin = min([p.z for p in points])
        self._xmax = max([p.x for p in points])
        self._ymax = max([p.y for p in points])
        self._zmax = max([p.z for p in points])

        # Compute facet-defining inequalities
        # - this is done by iterating over all triples of points that span some plane H
        # - if all points in 'points' lie on one side of H (including H itself), this defines a facet
        self._facets = set()
        for triple in itertools.combinations(points, 3):
            normal = basics.normal_vector_of_triangle(triple)

            if not normal.is_zero():
                max_scal_prod = max([normal * p for p in points])
                min_scal_prod = min([normal * p for p in points])

                if max_scal_prod - min_scal_prod == 0:
                    raise Exception("polytope low-dimensional")

                scal_prod = normal * triple[0]

                if max_scal_prod == scal_prod:
                    self._facets.add(Inequality3d(normal, scal_prod))
                elif min_scal_prod == scal_prod:
                    self._facets.add(Inequality3d(-normal, -scal_prod))

        # determine which points are vertices
        is_vertex = lambda p: len([f for f in self._facets if f.holds_with_equality(p)]) >= 3
        self._vertices = frozenset([p for p in points if is_vertex(p)]) # frozenset important for hashing

    def bounding_box(self):
        """Returns a tuple (xmin, ymin, zmin, xmax, ymax, zmax) containing the minimum and maximum coordinates of 'self'
        with respect to all coordinate directions."""
        return (self._xmin, self._ymin, self._zmin, self._xmax, self._ymax, self._zmax)

    def integer_points_in_bounding_rect_of_slice(self, height):
        """Yields a set S of integer points such that every integer point in 'self' with z-coordinate being equal to
        'height' is contained in S."""
        # First, we compute the intersection of 'self' with the set of points having z-coordinate being equal to
        # 'height'; note that every vertex of this set is
        # - a vertex of 'self' with z-coordinate being equal to 'height', or
        # - the intersection of z = 'height' with a line segment between some vertex above and some vertex below z =
        #   'height'
        above = [v for v in self._vertices if v.z > height]
        inside = [v for v in self._vertices if v.z == height]
        below = [v for v in self._vertices if v.z < height]

        points = inside
        if len(above) > 0 and len(below) > 0:
            for a in above:
                for b in below:
                    mu = Fraction(a.z - height, a.z - b.z)
                    points.append(a * (1 - mu) + b * mu)

        # By construction, 'self' intersected with z = 'height' is the convex hull of 'points'; now we compute a
        # bounding rectangle (xmin, ymin, xmax, ymax) of 'points'
        if len(points) > 0:
            xmin = basics.custom_ceil(min([p.x for p in points]))
            ymin = basics.custom_ceil(min([p.y for p in points]))
            xmax = basics.custom_floor(max([p.x for p in points]))
            ymax = basics.custom_floor(max([p.y for p in points]))

            # Finally, we return all integer points inside the bounding rectangle
            for x in range(xmin, xmax + 1):
                for y in range(ymin, ymax + 1):
                    yield Vector3d(x, y, height)

    def is_lattice_free(self):
        """Returns true iff 'self' does not contain any integer point in its interior."""
        for height in range(basics.custom_ceil(self._zmin), basics.custom_floor(self._zmax) + 1):
            for p in self.integer_points_in_bounding_rect_of_slice(height):
                if self.contains_in_interior(p):
                    return False

        return True

    def contains_in_interior(self, p):
        """Returns true iff 'p' is contained in the interior of 'self'."""
        for f in self._facets:
            if not f.is_strictly_satisfied(p):
                return False
        return True

    def contains(self, p):
        """Returns true iff 'p' is contained in 'self'."""
        for f in self._facets:
            if not f.is_satisfied(p):
                return False
        return True

    def ymin(self):
        """Returns the minimum y-coordinate of points in the polytope 'self'."""
        return self._ymin

    def ymax(self):
        """Returns the maximum y-coordinate of points in the polytope 'self'."""
        return self._ymax

    def zmax(self):
        """Returns the maximum z-coordinate of points in the polytope 'self'."""
        return self._zmax

    def integer_points(self, zmin = None, zmax = None):
        """
        Returns all integer points of 'self' with z-coordinate between 'zmin' (default: -inf) and 'zmax' (default: inf).
        """
        if zmin is None:
            zmin = basics.custom_ceil(self._zmin)
        if zmax is None:
            zmax = basics.custom_floor(self._zmax)

        for height in range(zmin, zmax + 1):
            for p in self.integer_points_in_bounding_rect_of_slice(height):
                if self.contains(p):
                    yield p

    def has_ld_greater_than(self, g):
        """Returns true iff 'self' has lattice diameter greater than 'g'."""
        for [p, q] in itertools.combinations(self.integer_points(), 2):
            if (p - q).gcd() > g:
                return True
        return False

    def _eq_(self, other):
        """Returns true iff 'self' and 'other' represent the same polytope."""
        return self.vertices == other.vertices

    def _hash_(self):
        return hash(self._vertices)

    def successive_minimum_at_most(self, value):
        """Returns true iff the first successive minimum of 'self' is at most 'value'."""
        # Note that the first successive minimum of P := 'self' is at most 'value' if and only if P - P contains a
        # non-zero point of 1/'value' * Z^3; setting 'value' =: a/b, this is equivalent to the fact that a * (P-P)
        # contains a non-zero integer point in b * Z^d
        value = Fraction(value)
        difference_body = Polytope3d([(v - w) * value.numerator for v in self._vertices for w in self._vertices
                                      if v != w])

        for p in difference_body.integer_points():
            if p.gcd() >= value.denominator:
                return True
        return False

    def vertices(self):
        """Returns a list of all vertices of 'self'."""
        return list(self._vertices)

    def __mul__(self, multiple):
        """For a scalar 'multiple', this returns 'multiple' * 'self'."""
        V=[v * multiple for v in self._vertices]
        return Polytope3d(V)

    def homo_copy(self, scaling, center):
        """Returns 'scaling' * 'self' + (1 - 'scaling') * 'center'."""
        V = [v * scaling + center * (1 - scaling) for v in self._vertices]
        return Polytope3d(V)

    def all_facets_blocked(self):
        """Returns true iff every facet of 'self' contains an integer point in its relative interior."""
        intpoints = list(self.integer_points())
        for f in self._facets:
            if not self.is_facet_blocked(f, intpoints):
                return False
        return True

    def is_facet_blocked(self, facet, points):
        """Returns true iff 'facet' contains one point of 'points' in its relative interior."""
        for p in points:
            if facet.holds_with_equality(p):
                if len([g for g in self._facets if g.holds_with_equality(p)]) == 1:
                    return True
        return False

    def is_lattice_free_extendable(self, points):
        """
        Returns true iff there exists a point p in 'points' \ 'self' such that conv('self' union {p}) is lattice-free.
        """
        for p in points:
            if not self.contains(p) and Polytope3d(self.vertices() + [p]).is_lattice_free():
                return True
        return False
\end{lstlisting}

\subsection*{scripts.py}

\begin{lstlisting}
"""Contains very basic functions and objects needed for geometric computations."""

import fractions

def gcd_three(a, b, c):
    """Returns the absolute value of the greatest common divisor of 'a', 'b', and 'c'."""
    return abs(fractions.gcd(a, fractions.gcd(b, c)))

def custom_floor(x):
    """Floor function that does not use floating point arithmetics when applied to fractions."""
    return x // 1

def custom_ceil(x):
    """Ceiling function that does not use floating point arithmetics when applied to fractions."""
    return -custom_floor(-x)

def normal_vector_of_triangle(points):
    """Returns a primitive nonzero vector 'v' in Z^3 that is perpendicular to the plane containing the
    (three-dimensional) points 'points[0]','poins[1]','points[2]'."""
    a = points[1] - points[0]
    b = points[2] - points[0]
    normal = a.cross(b)
    if normal.is_zero():
        return Vector3d(0,0,0)
    return normal.primitive_vector()

def volume_of_simplex(vertices):
    """Computes the volume of the 3-dimensional simplex whose vertices are 'vertices'."""
    assert len(vertices) == 4
    a = vertices[1] - vertices[0]
    b = vertices[2] - vertices[0]
    c = vertices[3] - vertices[0]
    return abs(fractions.Fraction(a.cross(b) * c, 6))

# -------------------------------------------------------------------------

class Vector3d(object):
    """Straight-forward implementation of basic functions on vectors with three components."""
    def __init__(self, x, y, z):
        """Creates a vector with coordinates ('x', 'y', 'z')."""
        self.x = x
        self.y = y
        self.z = z

    def __str__(self):
        return "[" + str(self.x) + ", " + str(self.y) + ", " + str(self.z) + "]"

    def __repr__(self):
        return str(self)

    def __eq__(self, other):
        return self.x == other.x and self.y == other.y and self.z == other.z

    def __ne__(self, other):
        return self.x != other.x or self.y != other.y or self.z != other.z

    def __hash__(self):
        return hash(str(self))

    def __add__(self, other):
        return Vector3d(self.x + other.x, self.y + other.y, self.z + other.z)

    def __sub__(self, other):
        return Vector3d(self.x - other.x, self.y - other.y, self.z - other.z)

    def __neg__(self):
        return Vector3d(-self.x, -self.y, -self.z)

    def __mul__(self, other):
        """If 'other' is a scalar, this returns the vector that arises from 'self' by multiplying each component with
        'other'. If 'other' is a vector, this returns the scalar product of 'self' and 'other'."""
        if type(other) is Vector3d:
            return self.x * other.x + self.y * other.y + self.z * other.z
        return Vector3d(self.x * other, self.y * other, self.z * other)

    def __div__(self, other):
        """Returns the vector that arises from 'self' by dividing each component by 'other'."""
        return Vector3d(self.x / other, self.y / other, self.z / other)

    def gcd(self):
        """Returns the greatest common divisor of the entries of 'self'."""
        return gcd_three(self.x, self.y, self.z)

    def is_zero(self):
        """Returns true iff all entries of 'self' are zero."""
        return self.x == 0 and self.y == 0 and self.z == 0

    def is_integral(self):
        """Returns true iff all entries of 'self' are integer numbers."""
        return all([type(t) is int for t in [self.x, self.y, self.z]])

    def primitive_vector(self):
        """Returns an integral vector that is parallel to 'self' and whose entries have gcd one."""
        x = fractions.Fraction(self.x)
        y = fractions.Fraction(self.y)
        z = fractions.Fraction(self.z)

        v = Vector3d(x.numerator * y.denominator * z.denominator, \
                     y.numerator * x.denominator * z.denominator, \
                     z.numerator * x.denominator * y.denominator)

        return v / v.gcd()

    def cross(self, other):
        """Returns the 3-dimensional cross product of 'self' and 'other'."""
        return Vector3d(self.y * other.z - self.z * other.y,
                        self.z * other.x - self.x * other.z,
                        self.x * other.y - self.y * other.x)
\end{lstlisting}

\end{document}